\documentclass[12pt,a4paper]{article}
\usepackage[utf8]{inputenc}
\usepackage[T1]{fontenc}
\usepackage{amsmath}
\usepackage{amsfonts}
\usepackage{amssymb}
\usepackage{amsthm}
\usepackage{authblk}
\usepackage{bbold}
\usepackage{makeidx}
\usepackage{pgf,tikz}
\usepackage{tikz}
\usepackage{mathrsfs}
\usepackage{graphicx}
\usepackage[all]{xy}
\usepackage{stmaryrd}
\usetikzlibrary{arrows}
\usepackage{hyperref}
\usepackage{accents}

\author{Frédéric Bihan\thanks{frederic.bihan@univ-smb.fr}}
\author{Tristan Humbert\thanks{tristanhumbert74@gmail.com}}
\author{Sébastien Tavenas\thanks{sebastien.tavenas@univ-smb.fr}}
\affil{Université Savoie Mont Blanc, CNRS}
\date{}
\title{New bounds for the number of connected components of fewnomial hypersurfaces} 

\theoremstyle{plain} 
\newtheorem{theo}{Theorem}[section]
\newtheorem*{theo*}{Theorem}

\newtheorem{cor}[theo]{Corollary}
\newtheorem{lem}[theo]{Lemma}
\newtheorem{prop}[theo]{Proposition}
\newtheorem{clm}[theo]{Claim}
\newtheorem{assump}[theo]{Assumption}
\newtheorem*{clm*}{Claim}

\theoremstyle{definition} 
\newtheorem{de}[theo]{Definition}
\newtheorem{ex}[theo]{Example}
\theoremstyle{remark} 
\newtheorem{rem}[theo]{Remark}

\DeclareMathOperator{\n}{\mathbb{N}}
\DeclareMathOperator{\p}{\mathbb{P}}
\DeclareMathOperator{\R}{\mathbb{R}}
\DeclareMathOperator{\RP}{\mathbb{P}\mathbb{R}}
\DeclareMathOperator{\Z}{\mathbb{Z}}
\DeclareMathOperator{\C}{\mathbb{C}}
\DeclareMathOperator{\CP}{\mathbb{P}\mathbb{C}}

\DeclareMathOperator{\A}{\mathcal{A}}

\let\S\relax
\DeclareMathOperator{\S}{\mathcal{S}}
\DeclareMathOperator{\signvar}{signvar}

\DeclareMathOperator{\RA}{\R^{\A}}

\DeclareMathOperator{\conv}{conv}
\DeclareMathOperator{\codim}{codim}
\DeclareMathOperator{\defeq}{\stackrel{\text{def}}{=}}
\DeclareMathOperator{\rk}{rk}

\DeclareMathOperator{\ch}{conv}
\DeclareMathOperator{\Log}{Log}
\DeclareMathOperator{\Diag}{Diag}

\newcommand{\bigzero}{\mbox{\normalfont\Large\bfseries 0}}

\newcommand{\mA}{{\mathrm{A}}}
\newcommand{\mAh}{\mA^{h}}
\newcommand{\mAF}{\mA_F}
\newcommand{\mAhF}{\mA^{h}_F}
\newcommand{\Ah}{\A^{h}}
\newcommand{\AF}{\A_F}
\newcommand{\AhF}{\A^{h}_F}
\newcommand{\mB}{\mathrm{B}}

\newcommand{\mC}{\mathrm{C}}
\newcommand{\mD}{\mathrm{D}}
\newcommand{\mG}{{\mathrm{G}}}
\newcommand{\mU}{{\mathrm{U}}}
\newcommand{\mV}{{\mathrm{V}}}
\newcommand{\mR}{{\mathrm{R}}}
\DeclareMathOperator{\MatrA}{\ensuremath{\mathcal{M}_{\A}}}
\DeclareMathOperator{\MatrAd}{\ensuremath{\mathcal{M}_{\A}^{*}}}
\DeclareMathOperator{\MatrF}{\ensuremath{\mathcal{M}_{\AF}}}

\begin{document}
\maketitle
\begin{abstract}
We prove that the number of connected components of a smooth hypersurface in the positive orthant of $\R^n$ defined by a real polynomial with $d+k+1$ monomials, where $d$ is the dimension of the affine span of the exponent vectors, is smaller than or equal to \(8 (d+1)^{k-1} 2^{\binom{k-1}{2}}\), improving the previously known bounds. We refine this bound for $k=2$ by showing that a smooth hypersurface defined by a real polynomial with $d+3$ monomials in $n$ variables has at most $\lfloor \frac{d-1}{2}\rfloor +3$ connected components in the positive orthant of $\R^n$. We present an explicit polynomial in $2$ variables with $5$ monomials which defines a curve with three connected components in the positive orthant, showing that our bound is sharp for $d=2$ (and any $n$).
Our results hold for polynomials with real exponent vectors.
\end{abstract}
\tableofcontents

\section{Introduction}

Descartes' rule of signs implies that the number of positive roots of a real univariate polynomial is bounded by a function which only depends on its number of monomials. {\em Fewnomials Theory} developed by Khovanski{\u{\i}}~\cite{Khovanskii} generalizes this result for multivariate real polynomials and more general functions.

In its monograph~\cite{Khovanskii}, Khovanski{\u{\i}} considered smooth hypersurfaces in the positive orthant \(\R_{>0}^n\) defined by real polynomials with \(n+k+1\) monomials. He showed (\cite{Khovanskii}, Sec. 3.14, Cor. 4) that the sum of the Betti numbers (and so the number of connected components) of such a hypersurface is at most
\[
(2n^2-n+1)^{n+k}(2n)^{n-1}2^{\binom{n+k}{2}}.
\]
The breakthrough was that the bound is only depending on \(n\) and on the number of monomials (and so holds true whatever is the degree of the considered polynomial).
In fact he proved his bound for hypersurfaces defined by real polynomials with real exponent vectors, that is, for zero loci of finite real linear combination of monomials $x^a$ where $a \in \R^n$.
Even if this upper bound was expected to be far from optimal, only few quantitative improvements have been found. 
The upper bound on the number of connected components was first improved by Li, Rojas, and Wang~\cite{LRW03}, 
and Bihan, Rojas, and Sottile~\cite{BRS08}. In the latter paper the authors showed that a smooth hypersurface in \((\R_{>0})^n\) defined by a real polynomial with \(d+k+1\) monomials with real exponent vectors spanning a \(d\)-dimensional affine span has fewer than
\[
\frac{e+3}{4} 2^{\binom{k+1}{2}}2^d d^{k+1}
\]
connected components. The approach was based on Bihan and Sottile's improved bound on the number of positive solutions of a system supported on few monomials~\cite{BS07}.

In a subsequent work, Bihan and Sottile~\cite{BS09} improved again their bound. They show that the sum of Betti numbers of such an hypersurface was bounded by	
\[
\frac{e^2+3}{4} 2^{\binom{k}{2}} \sum_{i=0}^d \binom{d}{i} i^k
\]
which already gives the better upper bound
\[
(e^2+3)2^{\binom{k}{2}}d^k2^{d-3}.
\]
The approach was based on the stratified Morse theory which was developed in~\cite{Goresky}.

More recently, the authors in~\cite{FNR17} give a new bound which is in particular not anymore exponential in \(d\). Taking our notations, they showed that if \(f\) is a \(d\)-variate real polynomial of support of cardinal \(d + k+1\) and not lying in an affine hyperplane, then, for generic coefficients, the corresponding hypersurface of the positive orthant $(\R_{>0})^n$ has no more than
\[
\binom{d+k+1}{2}+ 2^{\binom{k-1}{2}}(d+2)^{k-1}
\]
connected components. Furthermore, for \(k=2\), a sharper upper bound of \(\frac{(d+3)(d+2)}{2} + \lfloor \frac{d+5}{2}\rfloor\) is given. 

However, it seems that there is a problem in their proof since they claim (Proposition 3.4 in~\cite{FNR17}) that the number of non-simplicial faces of a \(d\)-dimensional polytope with \(d+k+1\) vertices is at most \(k\). In particular, a cube is a \(3\)-dimensional polytope (giving \(k=4\)) with already \(6\)
faces (the \(2\)-dimensional faces) which are non-simplicial. The situation happens to be even more problematic, since using Lawrence polytopes we can prove (see Proposition~\ref{Lawrence}) that the number of non-defective faces can be exponential in \(k\).\footnote{In fact in a personal communication with Maurice Rojas (one of the 
authors of~\cite{FNR17}), we learned that the authors are aware of the 
error. They presented a corrected version of their result at the MEGA 
2022 conference. It seems they will be releasing a revised version of 
their paper. However, they would now have a bound in 
\(O((k+1)(d+2)^{\max(0,(k^2+3k-1)/2)}\), which is exponential in \(k^2 
\log(d)\) and is therefore significantly worse than our bound which is 
exponential in \(k\max (k,\log d)\).}

\subsubsection*{Our results}

In this paper, we improve the previous upper bounds on the number of connected components of a smooth hypersurface of the positive orthant. We follow the approach developed by Gel'fand, Kapranov, and Zelevinsky~\cite{GKZ}. Let $\A$ be a finite set in $\R^{n}$ which spans an affine space of dimension \(d\). The cardinal of \(\A\) is written as \(d+k+1\). The parameter \(k\) is called the codimension of \(\A\). A \emph{generalized polynomial} with support in $\A$ is a (finite) linear combination $f=\sum_{a \in \A} c_a x^a$. We will often say that $f$ is a polynomial for shortness (even if the exponents vectors $a$ are real vectors) and we will sometimes say that $f$ is a $\A$-polynomial to indicate that its support is contained in $\A$. We will mostly take real coefficients $c_a$ and will denote by $V_{>0}(f)$ the zero locus of $f$ in the positive orthant of $\R^n$.
We give here our main results (see the cited results to get stronger versions with more precise hypotheses). 

Our new bound on the number of connected components of $V_{>0}(f)$ in the general case is again polynomial in \(d\) when \(k\) is fixed.
\begin{theo*}[Theorem~\ref{thm_generic_bound}]
	Let $\A$ be a finite set in $\R^{n}$ of dimension \(d\) such that $\codim \A= k$. 
	For any real polynomial $f=\sum_{a\in \A}c_a x^{a}$, we have
	\begin{equation*}
	b_{0}(V_{>0}(f)) \leq  1+k+\frac{e^2+3}{8(m+k+2)} \sum_{j=0}^{k-1} \binom{d+k+2}{k-j}  2^{\binom{j}{2}} (d+1)^{j}.
	\end{equation*}
	The bound can be replaced by the weaker but simpler expression
	\begin{equation*}
		b_{0}(V_{>0}(f)) \leq 8 (d+1)^{k-1} 2^{\binom{k-1}{2}}.
	\end{equation*}
\end{theo*}

In fact we have a better estimate where $d$ is replaced by the dimension $m \leq d$ of the basis of $\A$, see Theorem~\ref{thm_generic_bound}.
The approach of the proof is close to that of~\cite{FNR17}. We see the polynomial $f=\sum_{a\in \A}c_a x^{a}$ in a family 
$f_t=\sum_{a\in \A}c_a t^{h_a}x^{a}$ parametrized by a positive parameter $t$, where $h=(h_a)_{a \in \A}$ is a sufficiently generic collection of real numbers. Such polynomials are often called \emph{Viro} polynomials. The starting polynomial $f$ is obtained by setting $t=1$.
 We analyze the set $T$ of values $t \in \left]0,+\infty \right[ $ for which the hypersurface defined by $f_{t}$ admits a singularity in the positive orthant or "at infinity". A point $x$ with non-zero coordinates is a singular point of the hypersurface defined by $f_t$ if and only if it is solution to the \emph{critical system}
 \begin{equation}
 \label{critIntro}
 f_t=x_1\frac{\partial f_t}{\partial x_1} = \cdots =x_n\frac{\partial f_t}{\partial x_n}=0.
 \end{equation}
 A singularity at infinity is a singularity in the corresponding positive orthant of the hypersurface defined by the truncation of $f_t$ to some proper face of the Newton polytope $Q$ of $f$. Only \emph{non-simplicial} or, more generally, \emph{non-defective} faces need be taken into account. Indeed the \emph{discriminant variety} corresponding to a defective face is either empty or has codimension at least two and so can be avoided by a generic path $(f_t)_{t>0}$. This leads us to consider critical systems associated to non-defective faces of $Q$. When $h=(h_a)_{a \in \A}$ is  sufficiently generic, then only one singularity appears for each $t_0 \in T$, this singularity is an ordinary double point of the corresponding hypersurface (by a result of Forsg{\aa}rd \cite{For19}) and the number of connected components of $V_{>0}(f_t)$
 varies at most by one when $t$ passes through $t_0$. We are therefore reduced to estimate the cardinality of $T$, the number of connected components
 of $V_{>0}(f_t)$ for $t>0$ small enough and for $t>0$ large enough.
 
 In order to estimate the cardinality of $T$, we use matroid theory and correct the bound on the number of non-simplicial faces of a polytope which was given in~\cite{FNR17}. Then, for each such a face, we give an upper bound for the number of positive solutions of the corresponding critical system. The bound we use are essentially those obtained in~\cite{BS07} using so-called Gale dual systems. Interestingly enough, we also get sign conditions on the coefficients $c_a$ which are necessary for the existence of positive solutions of the critical system, equivalently, for the existence of singular points in the positive orthant for the corresponding hypersurface (see Proposition \ref{signs}).
 
 In order to estimate $b_0(V_{>0}(f_t))$ for $t>0$ small or large enough, we use the version in the positive orthant of Viro's patchworking Theorem (see, for instance, \cite{Vir84}, \cite{Vir06}, \cite{Ris92}, \cite{GKZ}) which allows to compute the topological type (and thus the number of connected components) of $V_{>0}(f_t)$ for $t>0$ small enough.
Though the original Viro's patchworking Theorem is stated for integer exponent vectors $\A$, in other words, for Laurent polynomials,
it is clear from the proof that the version in the positive orthant works equally well when $\A$ is a finite subset of $\R^n$ (see Proposition \ref{P:Patchwork}).
This enables us to prove the following intermediate result which we believe could be interesting for itself.
\begin{theo*}[Theorem~\ref{patchwork}] Assume that \(h=(h_a)_{a \in \A}\in \R^{\A}\) is generic enough. There exists $t_0>0$ such that for every $t\in \left]0,t_0\right]$, $f_t = 0$ has at most \(k+1\) connected components in $(\R_{>0})^n$.
If \(n \geq 2\) and \(k \geq 2\), then the bound can even be improved to \(k\).
\end{theo*}
Note that replacing $t$ by $1/t$ leads to the same bound for $t>0$ large enough.

Then, we focus on the codimension \(2\) case. When $k=2$ a finer analysis of the Gale dual of a critical system and arguments from matroid theory combined with Descartes bounds obtained in~\cite{BD} for the number of positive solutions of a system supported on a circuit enable us to prove the following result.
\begin{theo*}[Theorems~\ref{Fewnomial boundbis} and~\ref{boundusingcircuits}]
	If \(\A\) is a finite set of \(\R^n\) of dimension \(d\) and codimension \(2\), then $b_{0}(V_{>0}(f)) \leq \left\lfloor \frac{d-1}{2} \right\rfloor+3$.
\end{theo*}

This bound is optimal for \(d=2\), indeed we observe that the bivariate polynomial $f(x,y) = 1+ x^{4}-xy^{2}-x^{3}y^{2}+0.76x^{2}y^{3}$ has three connected components in its positive zero set (see Theorem~\ref{sharp n=2}).

\subsubsection*{Warm-up: the codimension \texorpdfstring{\(0\)}{0} and \texorpdfstring{\(1\)}{1} cases}
 
We detail the general strategy explained above a bit more now in the
codimension $0$ and codimension $1$ easy cases.
We decided to keep the objects intuitive in this warm-up to help the reader to understand the idea of the proof. The tools we use will be defined more formally in the next sections.

Let \(\A =\{a_0,\ldots, a_{d+k}\} \subset \R^n\) of dimension \(d\) and let \(f\) be a $\A$-polynomial. The case  \(k=0\) arises when $\A$ is the set of vertices of a simplex
and an appropriate monomial change of variables with real exponent vectors transforms the positive zero set of \(f\) into an affine hyperplane. Such a monomial change of coordinates gives rise to a diffeomorphism from the positive orthant to itself. In particular, the hypersurface defined by $f$ is never singular, and its positive part is either empty (when all coefficients of $f$ have the same sign) or diffeomorphic to the positive part of an affine hyperplane. 

A finite set $\A \subset \R^{n}$ for which there exists $a \in \A$ such that $\A \setminus \{a\}$ is contained in some affine subspace of dimension strictly smaller than $d=\dim \A$ is called \emph{a pyramid}. Note that if $\A$ has codimension $0$ then it is a pyramid. Pyramidal sets provide another important family of sets $\A$ with the property that the zero set of any polynomial $f=\sum_{a \in \A} c_{a}x^{a}$ is non-singular. This follows from the fact that if $\A$ is pyramidal then via a monomial change of coordinates $f$ can be transformed into a polynomial of the form $x_{n}+g(x_{1},\ldots,x_{n-1})$.

The case \(k=1\) has been treated in \cite{BHDR} and \cite{BRSte} following the general strategy used in this paper.
\begin{prop}
	Let \(\A =\{a_0,\ldots, a_{d+k}\} \subset \R^n\) of dimension \(d\). Let \(f\) be a $\A$-polynomial with non singular positive zero set \(V_{>0}(f)\).
	\begin{itemize}
		\item If \(k=0\) then \(V_{>0}(f)\) is either empty or connected and non-compact.
		\item If \(k=1\) then \(V_{>0}(f)\) has at most \(2\) connected components.
	\end{itemize}
\end{prop}

Write \(f(x) = \sum_{a \in \A} c_a x^a\).  Let \(h: \A \rightarrow \R\) be a generic function and consider the associated Viro polynomial \(f_t(x) = \sum_{a \in \A} c_at^{h_a} x^a\). Assume that the positive zero set \(V_{>0}(f)\)  is non-singular.

From Viro's patchworking Theorem the topology of \(V_{>0}(f_t)\) is well understood when \(t\) is small or large enough. We want to understand the topology of \(V_{>0}(f_t)\) when \(t=1\). To do that it is sufficient to get \(t\) moving and focus when the topology changes. Using stratified Morse theory for instance (see the monograph~\cite{Goresky}), it can be seen that the topology can change only when a singularity appears, either in the positive orthant or at "infinity". As already explained before, a singularity in the positive orthant corresponds to a positive solution of the critical system \(f_t = x_1 \partial f_t / \partial x_1 = \cdots = x_n\partial f_t /\partial x_n = 0\), while a singularity at infinity corresponds to a positive solution of the critical system associated to the truncated polynomial \(f_t^F = \sum_{a \in \A \cap F} c_a t^{h(a)} x^a\) (where \(F\) is a face of \(Q = \conv(\A)\)). To highlight this fact, let us see it on an example which was given by Forsg{\aa}rd, Nisse, and Rojas (Example 3.7 in~\cite{FNR17}). Taking \(\A = \{(0,0),(1,0),(2,0),(0,1),(0,2)\}\), the circle defined by \(f_0=(x+1/2)^2+(y-2)^2 -1=0\) intersects the positive orthant whereas the curve \(f_1 = (x+3/2)^2+(y-3/2)^2-1=0\) does not intersect the positive orthant (see Figure~\ref{exFNR}). However, the parameterized polynomial \( f_t = (x+1/2+t)^2+(y-(4-t)/2)^2-1\) connects \(f_0\) to \(f_1\) and the hypersurface defined by $f_{t}$ has no singular point in the positive orthant for \(t \in [0,1]\). In fact the change of topology comes from a singularity which appears in the \(y\)-axis. Indeed, if we consider the face $F$ of $Q$ given by the convex hull of \(\{(0,0),(0,1),(0,2)\}\), we can notice that the hypersurface defined by the corresponding truncated polynomial \(f_t^F = y^2 + (t-4)y +5t^2/4 -t +13/4\) has a singular point in the positive orthant for \(t=1/2\).
\begin{figure}\label{exFNR}
	\centering
	\begin{tikzpicture}
		 \begin{scope}[thick,font=\scriptsize]
			\draw [->] (-4,0) -- (3,0);
			\draw [->] (0,-2) -- (0,4) ;
			\foreach \n in {-3,...,-1,1,2,...,2}{%
				\draw (\n,3pt) -- (\n,-3pt)   node [below] {$\n$};
			}
		\foreach \n in {-1,...,-1,1,2,...,3}{%
		\draw (-3pt,\n) -- (3pt,\n)   node [right] {$\n$};
		}
		\end{scope}
		\draw [color=red] (-0.5,2) circle (1);
		\draw [color=green] (-1.5,1.5) circle (1);
		\draw [color=blue,dashed] (-1,1.75) circle (1);
	\end{tikzpicture}
	\caption{The red circle defined by \(f_0=0\) intersects the positive orthant and not the green one defined by \(f_1=0\). The blue circle is the zero set of \(f_{1/2}\).}

\end{figure}
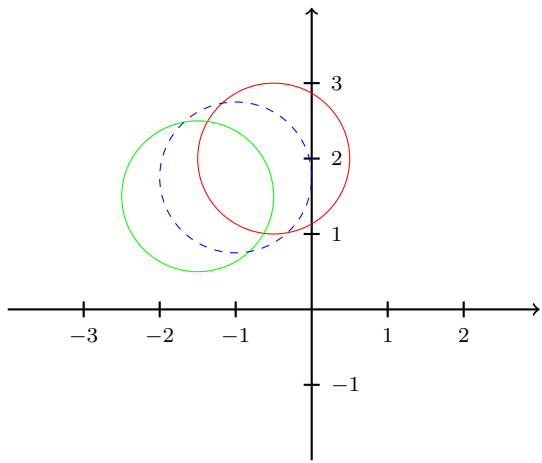

Let us go back for a moment to the codimension \(0\) case. Here, all intersections of $\A$ with faces of $Q$ have codimension $0$. In particular, the hypersurfaces defined by \(f_{t}\) and its truncations to faces of $Q$ do not have singular points in the corresponding positive orthants. Then the positive zero sets of polynomials $f_{t}$ are all homeomorphic. Consequently \(V_{>0}(f)\) is either empty or connected (and non-compact).

Let us consider now the codimension \(1\) case. Thus, the configuration \(\A\) admits a unique affine relation i.e.\ there exists a unique (up to multiplication by a scalar) vector \(b = (b_a) \in \R^{\A}\) such that \(\sum_{a \in \A} b_a =0\) and \(\sum_{a \in \A} b_a \cdot a =0\). Let \(\A^\prime\) be the support of \(b\) (defined by \(a \in \A^\prime\) if and only if \(b_a \neq 0\)). The configuration \(\A^\prime\) is known as a circuit. The convex hull $Q'$ of $\A'$ is a face of $Q$ and any triangulation of \(\A\) is induced by a triangulation of \(\A^\prime\). Furthermore, the triangulations of a circuit are also well known (see for example Proposition 1.2 in~\cite{GKZ}, in fact there are only two), and so, we know also the topology of \(V_{>0}(f_t)\), for \(t\) extremal, using again Viro's patchworking Theorem. More precisely, when \(t\) is extremal, if \(\A\) has an interior point, then \(V_{>0}(f_t)\) is either empty, or has one (compact or not) connected component (the choice only depends on the coefficients \(c_a\)). Otherwise, \(V_{>0}(f_t)\) contains \(0\), \(1\), or \(2\) non-compact connected component and no compact component. 

What happens now when we move \(t\)? Again, the topology can change only when a singularity appears. So let us consider all faces $F$ of $Q$ (including $Q$ itself) whose associated critical system might have a positive solution. It turns out that the only face $F$ of $Q$ such that $\A \cap F$ is not a pyramid is $Q ^\prime$, the convex hull of $\A'$. As noticed before, when $\A \cap F$ is a pyramid, the corresponding critical system has no solution.
So we just need to search for the positive solutions in \((x,t)\) of the system \(f_t^{Q^\prime} = x_1 \partial f_t^{Q^\prime}/ \partial x_1 = \cdots = x_n \partial f_t^{Q^\prime} / \partial x_n =0\). Let us emphasize here that we are reduced to studying the positive solutions of a system whose (exponents of) monomials form a set of codimension \(0\). However such a system has at most one positive solution \((t^\prime,x^\prime)\) (since up to monomial change of coordinates it is a linear system). Consequently there is no change in the topology in the interval \(0<t < t^\prime\),  neither in the interval \(t>t^\prime\). As \(V_{>0}(f)\) is smooth, it has to be homeomorphic to one of the two \(t\)-extremal cases. This finishes the proof of the proposition.

\subsubsection*{Outline of the paper}

We present basic tools used in this article in Section~\ref{Preliminaries}. This includes results about generalized discriminant varieties, Viro polynomials and critical systems. In Section~\ref{Critical systems}, we recall Gale duality for polynomial systems and
give bounds for the number of positive solutions of critical systems.
Section~\ref{S:Viro} is devoted to hypersurfaces constructed by the combinatorial patchworking construction. In Section~\ref{S:def} we present bounds for the number of non-defective faces of a polytope. In Section~\ref{Final bounds} we bring together all the previous results to obtain bounds for the number of connected components of an hypersurface in the positive orthant.
We conclude with an appendix containing basic results about monomial change of coordinates and technical proofs about generalized discriminant varieties.

\subsubsection*{Acknowledgements}
We thank Alicia Dickenstein, Elisenda Feliu, Maurice Rojas and Mate Telek for their interest in our work.
We want to thank Francisco Santos for pointing to us the example of Lawrence polytope as a polytope with many non-simplicial faces.

\section{Preliminaries}\label{Preliminaries}
	
\subsection{Generalized polynomials}
\begin{de} Let $\A$ be a non-empty finite subset of $\R^{n}$. We denote by \(\lvert \A\rvert\) its cardinal. The dimension $\dim(\A)$ of $\A$ is the dimension of the affine span of $\A$. The codimension of the set $\A$ is\footnote{Do not confuse with the codimension of the affine space generated by \(\A\). In fact, the terminology comes from the fact that it coincides with the codimension of the associated toric variety $X_{\A} \subset \CP^{|\A|-1}$.} the nonnegative integer $|\A|-\dim(\A)-1$.
\end{de}

We will usually denote by $d$ the dimension of \(\A\) and by $k$ its codimension.

\begin{ex}
The codimension of $\A$ is equal to $0$ if and only if $\A$ is the set of vertices of a simplex.
The elements of $\A$ are affinely dependent if and only if $\codim \A \geq 1$.
\end{ex}

For \(a \in \R^n\), we set \(x^a=x_1^{a_1}x_2^{a_2}\cdots x_n^{a_n}\).
We will call \emph{generalized polynomial} or simply \emph{polynomial} any finite linear combination of monomials $x^a$ for $a \in \R^n$.
Given a finite set $\A \subset \R^n$, the support of a polynomial $f=\sum_{a \in \A} c_a x^a$ is the subset of $\A$ of all $a$ such that $c_a \neq 0$.
Let us identify \(\R^{\A}\) with the space of real polynomials with supports contained in \(\A\), that is to say, we identify \(c=(c_a)\in \R^{\A}\) with the polynomial
\[
f(x) = \sum_{a\in \A}c_a x^{a}.
\]

Let $f \in  \R^{\A}$. We denote by \(V_{>0}(f)\) the zero set of $f$ in $(\R_{>0})^n$.
A numbering $\{a_{0},\ldots,a_{d+k}\}$ of $\A$ being given, the associated \emph{exponent matrix} is
\[
\mA = \begin{pmatrix}
	1 & 1 & \ldots & 1\\
	a_0 & a_1 & \ldots & a_{d+k}
\end{pmatrix} \in \R^{(n+1) \times (d+k+1)}.
\]
The convex hull $ \conv(\A)$ of $\A$ is a polytope that we will often denote by $Q$. The dimension of a polytope is the dimension of its affine span, so $\dim Q=\dim \A = \rk{\mA}-1$. Therefore $\dim \A=d$ if and only if the matrix \(\mA\) has rank \(d+1\). 
We will also use the matrix \(\hat{\mA} \in \R^{n \times (d+k+1)}\) which is the matrix \(\mA\) without its first row of one's.

The support of a polynomial system is the set of exponent vectors appearing with a non-zero coefficient in some equation, equivalently, this is the union of the supports of its polynomials. A real polynomial system with support in $\A=\{a_{0},\ldots,a_{d+k}\}$ can be written as
\[
\left\{
	\begin{array}{ccc}
		c_{1,0} x^{a_0} + c_{1,1} x^{a_1} + \cdots + c_{1,d+k} x^{a_{d+k}} & = &0 \\
		\vdots & \vdots  & \vdots \\
		c_{p,0} x^{a_0} + c_{p,1} x^{a_1} + \cdots + c_{p,d+k} x^{a_{d+k}} &=&0 \\
	\end{array}
\right.
\]
 where \(\mC = (c_{i,j}) \in \R^{p \times (d+k+1)}\).
 The matrix \(\mC\) is called the {\em coefficient matrix} of the system.
 
 \subsection{Gale duality and matroid theory}
 
Let \(\mU\) be a matrix in \(\R^{p \times q}\) of rank \(r\). A Gale dual of \(\mU\) is a matrix \(\mV \in \R^{q \times (q-r)}\) of full rank such that  \(\mU\cdot \mV\) equals \(0\) (the columns of $\mV$ form a basis of $\ker \mU$). We can notice that \(\mV\) is defined  up to multiplication on its right by a matrix in \(\text{GL}_{q-r}(\R)\). 

In the following, we will mainly use Gale duality for the exponent matrix \(\mA\) and for the coefficient matrix \(\mC\). These Gale dual matrices are usually denoted respectively \(\mB\) and \(\mD\).
This duality already appears in matroid theory.
For the reader unfamiliar with matroid theory, most of the paper can be read without knowledge on it. However, the viewpoint of the matroid theory gives a more enlightening vision of the considered objects. More information about matroid theory could be found in~\cite{Welsh,Oxley}\footnote{In fact, the matroids we consider in this paper are even oriented matroids. More information on them can be found in~\cite{BVSWZ} or in Section 6 of~\cite{Ziegler}. However, we let this specialization aside since we will not need it.}. We present here the links with this theory.

The affine dependences give to \(\A\) a structure of {\em matroid} (a subset \(\mathcal{I} \subset \A\) is independent in the matroid if and only if \(\mathcal{I}\) is an affinely independent subfamily of \(\A\)). This matroid will be denoted \(\MatrA\).

A \emph{circuit} is a subset of \(\A\) which is affinely dependent and minimal amongst such sets.
Equivalently, a circuit is a set which has codimension $1$ and such that all its proper subsets have codimension \(0\).

Moreover, if \(\mB\) is a matrix Gale dual to \(\mA\), then the {\em dual of the matroid \(\MatrA\)} (called \(\MatrAd\)) is exactly the rows matroid of \(\mB\) (that is to say, a subset \(\S\) of rows of \(\mB\) are independent if and only if they are linearly independent). 
 
 Let $\A$ be a finite set in $\R^{n}$ with convex hull $Q$. We will call {\em face} of $\A$ any subset $\AF=\A \cap F$ for some face $F$ of $Q$.
We can consider the restriction \(\MatrF\) of \(\MatrA\) to the subset \(\AF\). So, the codimension of \(\AF\) is just the rank of \((\MatrF)^{*}\).

A set $\A$ in $\R^{n}$ of dimension \(d\) is a {\em pyramid} if there exists $a \in \A$ such that $\A \setminus \{a\}$ is contained in some affine space of dimension \(d-1\). More generally, we will say that \(\A\) is a {\em pyramid over \(\S\) }(where \(\S \subsetneq \A\)) if \(\S\) belongs to an affine space of dimension \(d-\lvert \A\rvert +\lvert \S\rvert\). It means that any circuit of \(\A\) is in fact a circuit of \(\S\). By duality, this is equivalent to the fact that the rows \(\{b_\alpha \mid \alpha \in \A\setminus \S\}\) of \(\mB\) are \(0\). Consequently, from any set \(\A\) we can extract a unique subset \(\A' \subset \A\) such that \(\A\) is a pyramid over \(\A^\prime\) and \(\A'\) is not a pyramid (by duality, it just means we consider only the non-zero rows of \(\mB\)). The set \(\A'\) will be called {\em the basis} of \(\A\).

\begin{rem}\label{rem:npfaces}
	If \(F\) is a face of $Q$ and if \(\AF\) is non-pyramidal, then any element of \(\AF\) belongs to some circuit and so is in \(\A'\). Consequently, any non-pyramidal face \(\AF\) of $\A$ is in fact a face of $\A'$
(i.e.\ there exists a face $F'$ of the convex hull of $\A'$ such that  \(\AF=\A' \cap F'\)).
\end{rem}

The following lemma is elementary
\begin{lem}
	\label{codim-pyramid}
	Let $\A$ be a finite set in $\R^{n}$. For any proper subset $\S$ of $\A$, we have $\codim  \S \leq \codim \A$. Moreover, we have $\codim  \S= \codim \A$ if and only if $\A$ is a pyramid over $\S$.
\end{lem}
In particular a set \(\A\) has same codimension as its basis.

\subsection{A-discriminant}
\label{S:Adiscr}
We first assume that $\A$ is a finite subset of $\Z^n$. A polynomial with support in $\A$ is a classical \emph{Laurent} polynomial.
The theory of the \(\A\)-discriminants has been studied in detail in the book of Gel'fand, Kapranov, and Zelevinsky~\cite{GKZ}. We now recall basic definitions and properties of this theory.

\begin{de}
Given any \(\A \subset \Z^n\),  the \emph{\(\A\)-discriminant variety}
	\(\nabla_{\A} \subset \C^{\A}\) is the Zariski closure of the set of all \(c \in (\C^{*})^{\A}\) such that the hypersurface \(\sum_{a \in \A} c_ax^{a}=0\) has a singular point in \((\C^{*})^n\). 
	\end{de}
When \(\nabla_{\A}\) has codimension \(1\), we define \(\Delta_{\A} \in \Z[c_a, a \in \A]\) -- the \emph{\(\A\)-discriminant} -- to be the unique (up to sign) irreducible defining polynomial of \(\nabla_{\A}\). Otherwise, when \(\nabla_{\A}\) has codimension at least \(2\) or is empty, we set \(\Delta_{\A}\) to the constant \(1\).  
	
	If \(F\) is a face of \(Q=\conv(\A)\), we define similarly \(\nabla_{\AF}\) and \(\Delta_{\AF}\). In the case where \(\Delta_{\AF} \equiv 1\), we say that the face \(F\), or the set $\A_{F}$, is \emph{defective}. 
	For every \(f \in \R^{\A}\) and every face \(F\) of \(Q\), we denote by \(f^F\) the polynomial {\em truncated} to the face \(F\):
	\[
		f^F = \sum_{a \in \AF} c_a x^a.
	\]
	Its coefficients are $(c_{a})_{a \in \A_{F}}$.
We set \(\tilde{\nabla}_{\AF} = \{ c \in \RA \mid (c_{a})_{a \in \A_{F}} \in \nabla_{\AF}\}\).
Finally, we consider
	\[
		\nabla = \bigcup_{F\text{ face of }Q} \tilde{\nabla}_{\AF}
	\]
	and we denote by \(\R_{\text{gen}}^{\A}\) the set \(\RA \setminus \nabla\).

We will need more notations and results from the book \cite{GKZ}.
The toric variety $X_{\A}$ associated with $\A=\{a_{0},\ldots,a_s\} \subset \Z^{n}$ is the Zarisky closure in $\CP^{s}$ of the image of the map $\varphi_{\A}: (\C^{*})^{n} \rightarrow \CP^{s}$ defined by $\varphi_{\A}(x)=(x^{a_{0}}:\cdots:x^{a_{s}})$.
Define $X_{\A}^{\circ}(\R_{>0}) \subset X_{\A}$ as the image of the positive orthant $(\R_{>0})^{n}$ under the map $\varphi_{\A}$ and let $X_{\A}(\R_{>0})$ denote the closure of $X_{\A}^{\circ}(\R_{>0})$  in $X_{\A}$
(see \cite{GKZ}, Chapter 11). 
Consider the stratification of $X_{\A}(\R_{>0})$ given by its intersections with the subtoric varieties $X_{F}$ given by the faces of the convex hull $Q$ of $\A$.
The moment map associated with $\A$ is the map $\mu: X_{\A} \rightarrow Q$ which extends the map $\varphi_{\A}((\mathbb{C}^*)^n) \to Q, z \mapsto \frac{\sum_{a \in \A} |z_a| \cdot a}{\sum_{a \in \A} |z_a|}$.
The restriction of $\mu$ to 
$X_{\A}^{\circ}(\R_{>0})$ induces a diffeomorphism onto the relative interior of $Q$ (see, for instance, \cite[Chapter 6]{GKZ} or
\cite{Ris92}).
\begin{prop}[\cite{GKZ}, Theorem 5.3, page 383]
\label{moment map}
The moment map $\mu$ induces a stratified map $X_{\A}(\R_{>0}) \rightarrow Q$
sending $X_{\A_{F}}(\R_{>0})$ to $F$ for each face $F$ (including $Q$). Moreover the restriction of $\mu$ to $X_{\A_{F}}^{\circ}(\R_{>0})$ is a diffeomorphism onto the relative interior $\mbox{relint}(F)$ of $F$ for each face $F$ of $Q$ (including $Q$).
\end{prop}

Let $f\in \R^{\A}$ be a polynomial with support $\A$. For any face $F$ of $Q$, define the {\it chart}
$C_{F}^{\circ}(f^{F}) \subset \mbox{relint}(F)$ as the image of the hypersurface defined by $f^{F}$ in  $X_{\A_{F}}^{\circ}(\R_{>0})$ by the moment map associated with $\A_{F}$. By Proposition \ref{moment map}, this hypersurface and $C_{F}^{\circ}(f^{F})$ are diffeomorphic.
Furthermore, let $C_{F}(f^{F})$ be the closure of $C_{F}^{\circ}(f^{F})$ in $F$.

\begin{cor}
\label{strata for Z}
The space $C_{Q}(f)$ is a stratified space, the strata are given by the intersections with the faces of $Q$. We have $C_{Q}(f) \cap \mbox{relint}(F)=C_{F}^{\circ}(f^{F})$ and $C_{Q}(f) \cap F=C_{F}(f^{F})$ for any face $F$ of $Q$.
\end{cor}

We have assumed at the beginning of this subsection that $\A \subset \Z^{n}$. Since we are only interested by zero sets of \(\A\)-polynomials  in the positive orthant, we might relax this assumption and consider finite sets $\A \subset \R^{n}$. Then, \(\A\)-polynomials are in general not classical Laurent polynomials, they are sometimes called {\em generalized polynomials} in the literature.
This leads to the theory of irrational toric varieties which have been studied in several papers~\cite{PES},~\cite{CGDS},~\cite{PirSot}.
Allowing $\A$ to be a subset of $\R^{n}$ produces more general results, but it also requires a slight modification of the definitions above. Assume now that $\A$ is a finite set in $\R^n$ which is not contained in $\Z^{n}$. Though the above map $\varphi_{\A}$ is no longer well-defined for real exponent vectors, its restriction to the positive $(\R_{>0})^{n}$ is well-defined and it makes sense to set again $X_{\A}^{\circ}(\R_{>0})=\varphi_{\A}((\R_{>0})^{n})$.
Notice that if \(\A\) has dimension \(n\) then \(\varphi_{\A}\) induces a diffeomorphism between \((\R_{>0})^{n}\) and $X_{\A}^{\circ}(\R_{>0})$.
The space $X_{\A}(\R_{>0})$ is then the image in $\RP^{s}$ of the closure of $X_{\A}^{\circ}(\R_{>0})$ in the usual topology on $\R_{\geq 0}^{\A}=\R_{\geq 0}^{s+1}$. Denote by \(\mu\) the map $X_{\A}(\R_{>0}) \rightarrow Q, z \mapsto \frac{\sum_{a \in \A} |z_a| \cdot a}{\sum_{a \in \A} |z_a|}$, which will be called generalized moment map. This map is a homeomorphism \cite[Theorem 2.2]{PES} (see also \cite{PirSot}, Section 4.4) which restricts to a diffeomorphism between $X_{\A}^{\circ}(\R_{>0})$ and the relative interior of $Q$.
We define charts of generalized polynomials as in the classical case: for any $\A$-polynomial and any face $F$ of $Q$, the chart $C_{F}^{\circ}(f^{F}) \subset \mbox{relint}(F)$ is the image of the hypersurface defined by $f^{F}$ in  $X_{\A_{F}}^{\circ}(\R_{>0})$ by the map \(\mu\).
It follows that Proposition~\ref{moment map} and Corollary~\ref{strata for Z} still hold true when $\A \subset \R^{n}$.

We also need to modify the definition of discriminant varieties. Clearly, the sets will not be algebraic varieties anymore. However, they will still have a ``tame topology'' since they will be definable in a o-minimal expansion of \(\R\) (see~\cite{vD98,Cos00} for some nice expositions on the o-minimal theory). Let \(S \subseteq \R\).
By adding to the ordered field of real numbers \(\overline{\R} = \langle\R, +, \times, >,0,1\rangle\) the power functions \(x^r\) (with \(r \in S\)) defined by
\begin{align*}
	x & \mapsto \begin{cases} x^r & \text{if }x>0 \\ 0 & \text{otherwise,} \end{cases}
\end{align*}
we obtain the o-minimal expansion \(\overline{\R}^S\)~\cite{Mil94}.
We choose not to directly follow the approach with {\em exponential sums} as in~\cite{For19}, see also~\cite{FNR17}, since, although \(\R_{\text{exp}}\) is also o-minimal~\cite{Wil96}, \(\overline{\R}^S\) is a reduct of \(\R_{\text{exp}}\) and so has nicer properties  (for example \(\overline{\R}^S\) is polynomially bounded~\cite{Mil94}).

In the following, we will take as for \(S\) the finite set of coordinates of the elements of \(\A\). 
Hence, the \(\A\)-polynomials are definable functions in \(\overline{\R}^S\).

\begin{de}
	Given any \(\A \subset \R^n\),  let \(\accentset{\circ}{\nabla}_{\A}\) be the set of all \(c \in {(\R^{*})}^{\A}\) such that the hypersurface
	$\sum_{a \in \A}c_{a} x^a=0$ has a singular point in \((\R_{>0})^n\). The (generalized)  \emph{positive \(\A\)-discriminant variety}
	\(\nabla_{\A} \subset \R^{\A}\) is the Euclidean closure of \(\accentset{\circ}{\nabla}_{\A}\).
\end{de}
For \(c \in \RA\) denote by \(f_c\) the function defined by \(\sum_{a\in \A} c_a x^a\). Notice that \(\nabla_{\A}\) is the closure of the projection on the first \(\lvert \A\rvert\) coordinates of the definable set
\begin{equation*}
	\mathcal{D}_{\A} = \left\{(c,x)\in (\R^*)^{\A}\times(\R_{>0})^n \Bigm\vert 
	f_c(x)=x_1\frac{\partial f_c}{\partial x_1}(x) = \cdots = x_n\frac{\partial f_c}{\partial x_n}(x) = 0\right\}.
\end{equation*} 
Hence, \(\nabla_{\A}\) is a definable set. 

We consider the Horn-Kapranov uniformization~\cite{Kap91}. This parametrization still holds for exponential sums as it was shown in~\cite[Thm 1.7]{RR}. We follow, in the real case, the standard deduction of the Horn-Kapranov uniformization. Assume that \(\A\) is not pyramidal since otherwise \(\nabla_{\A}\) is empty. Let \(\mB\) be a Gale dual of the exponent matrix \(\mA\). Any row \(\mB_a\) of \(\mB\) gives rise to a non-zero linear form $y \mapsto \langle B_a,y \rangle$ of \(\R^k\). Let \(\mathcal{H}_{\mB} \subset \R^k\) be the hyperplane arrangement given by the union of the kernels of these linear forms. We have
\begin{align*}
c \in \accentset{\circ}{\nabla}_{\A} 
& \iff \exists x \in (\R_{>0})^n, \ A \cdot \left(c_a x^a\right)_{a \in \A} = 0 \\
&\iff \exists x \in (\R_{>0})^n \exists y \in (\R^k \setminus \mathcal{H}_{\mB}) \forall a \in \A,
\ c_a = x^{-a} \langle B_a,y \rangle
\end{align*}
where \(\left(c_a x^a\right)_{a \in \A}\) is given as a column vector.
So, \(\accentset{\circ}{\nabla}_{\A}\) is the image of \((\R_{>0})^n \times (\R^k \setminus \mathcal{H}_{\mB})\) by the previous definable function.

\begin{theo}[Kapranov, Rojas-Rusek]
\label{T:Horn}
    The set \(\accentset{\circ}{\nabla}_{\A}\) is parametrized by the map
\( \Phi_{\A} : (\R_{>0})^n \times (\R^k \setminus \mathcal{H}_{\mB}) \to (\R^*)^{\A}, 
\quad \Phi_{\A}(x,y) = \left( x^{-a} \langle B_a,y\rangle \right)_{a\in \A}\).
\end{theo}
Note that \(\accentset{\circ}{\nabla}_{\A}\) does not change if we apply an affine transformation of $\R^n$ to $\A$. Indeed, this amounts to multiplying a $\A$-polynomial $f$ by a monomial and doing a monomial change of coordinates. These operations preserve the fact that \(V_{>0}(f)\) is singular (see Section~\ref{monomial}). Thus if $d <n$ we might send $\A$ to a $d$-coordinate subspace by an affine transformation so that the parametrization in Theorem \ref{T:Horn} becomes a parametrization
\begin{equation}
\label{E:para}
\Phi_{\A} : (\R_{>0})^d \times (\R^k \setminus \mathcal{H}_{\mB}) \to (\R^*)^{\A}  
\end{equation}
(the ambient dimension $n$ is replaced by the dimension $d$ of $\A$).
In particular, this shows that the codimension of \(\nabla_{\A}\) is at least \(1\).
We will mostly use a parametrization like in \eqref{E:para} later on.
We say that a face \(F\), or the set $\A_{F}$, is \emph{defective} if \(\nabla_{\AF}\) has codimension at least 2 (see \cite{For19}, Section 6). Then we keep the other definitions  given above for the case $\A\subset\Z^{n}$. Namely, for any face $F$, define \(\tilde{\nabla}_{\A_F} = \{ c \in \RA \mid (c_{a})_{a \in \A_{F}} \in \nabla_{\A_F}\}\). Set
	\[
		\nabla = \bigcup_{F\text{ face of }\A} \tilde{\nabla}_{\AF}
	\]
	and denote by \(\R_{\text{gen}}^{\A}\) the definable set \(\RA \setminus \nabla\).
By~\cite[Thm. 3.9]{Cos00}, \(\R_{\text{gen}}^{\A}\) is partitioned into a finite number of definable open connected components which are arcwise connected.

\begin{prop}
\label{isotopic}
If $f_{0}$ and $f_{1}$ are $\A$-polynomials in the same connected component of $\R_{\mathrm{gen}}^{\A}$, then $V_{>0}(f_{0})$ and $V_{>0}(f_{1})$ are homeomorphic.
\end{prop}
\begin{proof}
Let us denote by \(c_0\) and \(c_1\) the coefficients of \(f_0\) and \(f_1\).
Consider a smooth path $(c_{t})_{t \in [0,1]}$ from $c_{0}$ to $c_{1}$, where the points $c_{t}=(c_{t,a})_{a \in \A}$ for $t \in [0,1]$ are  contained in the same connected component of $\R_{\mathrm{gen}}^{\A}$.
Define \(\hat{C} = \{ (x, t) \in Q\times [0,1] \mid x \in C_Q(f_t)\}\) where $f_{t}(x)=\sum_{a \in \A} c_{t,a}x^{a}$. Let $\pi: \hat{C} \rightarrow [0,1]$ denote the projection onto the $t$-coordinate. Then for each $t \in [0,1]$ we have $\pi^{-1}(t)=C_{Q}(f_{t}) \times \{t\}$, and furthermore $\pi^{-1}(t) \cap (F \times \{t\})=C_{F}(f_{t}^{F})\times \{t\}$ for any face $F$ of $Q$. We prove that $\hat{C}$ is a compact stratified space, in fact a manifold with corners, the strata being given by the intersections with the relative interiors of faces of $Q \times [0,1]$ and that $\pi$ is a stratified function with no critical point. We are reduced to prove that each stratum is a smooth manifold and that the restriction of $\pi$ to the stratum is a smooth function without critical point (our reference for stratified manifolds and for manifolds with corners is~\cite{Goresky}).
Consider a face of the form $F \times [0,1]$, where $F$ is a face of $Q$. 
Replacing $f$ by $f^F$ we can assume without loss of generality that $F=Q$. Consider the stratum $Z$ given
by the intersection of $\hat{C}$ with $\mbox{relint}(Q) \times ]0,1[$. Using the map $\varphi_{\A}$ and the corresponding (generalized) moment map as above,
we get that $Z$ is diffeomorphic to $\{(x,t) \in \R_{>0}^d \times ]0,1[, f_t(x)=0\}$. Since $t \mapsto (c_{a,t})_a$ and $x \mapsto f_t(x)$ are smooth functions, the map  $F:\R_{>0}^d \times ]0,1[ \rightarrow \R, (x,t) \mapsto f_t(x)$ is smooth. Moreover, the gradient of $F$ at a point $(x,t) \in \R_{>0}^d \times ]0,1[$ such that $f_t(x)=0$ is non-zero since the gradient with respect to the variable $x$ is non-zero due to the fact that the path $(c_{t})_{t \in [0,1]}$ is contained in one connected component of $\R_{\mathrm{gen}}^{\A}$ (it does not intersect the $\A$-discriminant variety). Thus $Z$ is a smooth manifold. Moreover, the latter fact also implies that the restriction of $\pi$ to $Z$ has no critical point.

Finally, since $\pi$ is a stratified function with no critical point, first Morse Lemma for manifolds with corners yields that $C_{Q}(f_{0}) \times \{0\}$ and $C_{Q}(f_{1}) \times \{1\}$ are homeomorphic.
\end{proof}

Before stating Proposition~\ref{Jens}, let us observe the following fact proved in  Appendix~\ref{sec:proof_intersection_nablas}.
\begin{lem}
\label{lem:intersection_nablas}
	If \(F_1\) and \(F_2\) are two distinct faces of \(Q=\ch(\A)\), then the set \(\tilde{\nabla}_{\A_{F_1}}\cap\tilde{\nabla}_{\A_{F_2}}\) has codimension at least \(2\) in \(\RA\).
\end{lem}

It is useful to consider the following sets. For any face $F$, define
\({\overset{\frown}{\nabla}}_{\A_{F}} = \{ c \in \RA \mid (c_{a})_{a \in \A_{F}} \in \accentset{\circ}{\nabla}_{\AF}\}\) and set
	\[
		\accentset{\circ}{\nabla} = \bigcup_{F\text{ face of }\A} {\overset{\frown}{\nabla}}_{\A_{F}}.\]
It follows that \(\tilde{\nabla}_{\A_F}\) is the Euclidean closure of \({\overset{\frown}{\nabla}}_{\A_{F}}\) and $\nabla$ is the Euclidean closure of $\accentset{\circ}\nabla$.

A point \(f\) of \(\nabla\) is called {\em smooth} if there exists a unique face \(F\) of \(Q\) such that \(f \in {\overset{\frown}{\nabla}}_{\A_{F}}\)  and \(f\) is a smooth point of \({\overset{\frown}{\nabla}}_{\A_{F}} \).
By~\cite{For19} (Theorem 3.5), the hypersurface $f^F=0$ has then a unique real singular point which is a non-degenerate double point in $X_{\A_{F}}^{\circ}(\R_{>0})$ (we can assume that $X_{\A_{F}}^{\circ}(\R_{>0})=\R_{>0}^{\dim F}$ sending $\A_{F}$ to a coordinate subspace by an affine transformation if necessary).
The previous lemma ensures that smooth points are generic in \(\nabla\) as soon as there is a non-defective face.

\begin{prop}
\label{Jens}
Assume that $f_{0}, f_{1} \in \R_{\mathrm{gen}}^{\A}$ are connected by a smooth path $c_t=(f_{t})_{t \in [0,1]} \subset \R^{\A}$ which intersects $\nabla$ at only one point $f_{t_{0}}$ which is a smooth point of $\nabla$. Assume furthermore that $\frac{\partial f_{t}^{F}}{\partial t}(x_0,t_0)\neq 0$, where $x_0$ is the singular point of $f_{t_{0}}^F=0$ contained in $X_{\A_{F}}^{\circ}(\R_{>0})$.
Then we have $|b_{0}(V_{>0}(f_0))-b_{0}(V_{>0}(f_1))|\leq 1$.
\end{prop}
\begin{proof}
The proof is similar to the proof of Proposition~\ref{isotopic} using $\hat{C}= \{ (x, t) \in Q\times [0,1] \mid x \in C_Q(f_t)\}$.
The first condition ensures that the stratified function $\pi$ has only one critical point $(z_0,t_0)$ which is non-degenerate and contained in the stratum $Z$
corresponding to $F$. The second condition implies that $Z$ is smooth at $(z_0,t_{0})$ and thus $\hat{C}$ is a manifold with corners.
Consequently, by Proposition~\ref{isotopic} there may be a topological change in the level sets of $\pi$ only around the critical value $t_{0}$. 
Consider a neighbourhood of the critical point $(z_0,t_{0})$ of the cylindrical form $V \times I$ (where \(I = [t_0-\varepsilon,t_0+\varepsilon]\) with \(\varepsilon>0\) small enough) and let $U ={(\mu \circ \varphi_{\A})}^{-1}(V) \subset \R_{>0}^{n}$. 
If $F=Q$,  it is sufficient to look directly at the hypersurfaces $V_{>0}(f_{t}) \subset \R_{>0}^{n}$ for $(x,t) \in U \times I$. Then the proof goes as the proof of Theorem 3.14 in~\cite{FNR17}. Shrinking $V$ and $I$ if necessary, the hypersurfaces $V_{>0}(f_{t}) \cap U$  are homeomorphic to the level sets of a non degenerate quadratic form (given by the Hessian of $f_{t_{0}}$ at the singular point), from which we get $\lvert b_{0}(V_{>0}(f_{t_0-\varepsilon}) \cap U)\rvert \leq 2$ and $\lvert b_{0}(V_{>0}(f_{t_0+\varepsilon}) \cap U)\rvert\leq 2$. Let \(i \in \{t_0-\varepsilon,t_0+\varepsilon\}\). Since a connected component of $V_{>0}(f_i) \cap U$ is contained in at most one connected component of  $V_{>0}(f_i)$, there are at most two connected components of $V_{>0}(f_i)$ which intersect \(U\). Moreover, for $\sigma \in \{\pm 1\}$, if the intersection of \(V_{>0}(f_{t_0+\sigma \epsilon})\) with \(U\) is empty (for \(U\) small enough), then the signature of the above quadratic form is either \((n,0)\) or \((0,n)\) and so \(V_{>0}(f_{t_0-\sigma \epsilon})\) has only one compact connected component more (lying inside \(U\)). In other cases, the result follows from the fact that for both \(i \in \{t_0 \pm \varepsilon\}\), there are only \(1\) or \(2\) connected components of \(V_{>0}(f_i)\) which intersect \(U\).

Assume now that $F$ is a proper face of $Q$. Recall that $C_{Q}(f_{t}) \cap F=C_{F}(f_{t}^{F})$ (by Corollary~\ref{strata for Z}). Thus the previous result applied to $f_{t}^{F}$ gives
$|b_{0}(C_{Q}(f_{t_0-\varepsilon}) \cap F \cap V)-b_{0}(C_{Q}(f_{t_0+\varepsilon}) \cap F \cap V)| \leq 1$ (shrinking $V$ if necessary).
But since $\{f_{i}^{F}=0\}$, $i=t_0\pm \varepsilon$, has no singular point in $X_{\A_{F}}^{\circ}(\R_{>0})$, each hypersurface $\{f_{i}=0\} \subset X_{\A}(\R_{>0})$ intersects $X_{\A_{F}}^{\circ}(\R_{>0})$ transversally along $\{f_{i}^{F}=0\}$. Thus $C_{Q}(f_{i})$ intersects transversally  $\mbox{relint}(F)$ along $C_{F}^{\circ}(f_{i}^{F})$, and thus a connected component of $C_{Q}(f_{i}) \cap V \cap F$ is contained in the closure of at most one connected component of $C_{Q}^{\circ}(f_{i})$. As a consequence, reasoning as in the case \(F=Q\) yields $|b_{0}(C_{Q}^{\circ}(f_{0}))-b_{0}(C_{Q}^{\circ}(f_{1}))| \leq 1$ and thus the desired inequality.
\end{proof}

\begin{rem}
\label{R:no restrictive}
We will apply Proposition \ref{Jens} to Viro polynomials $f_t(x)=\sum_{a \in \A} c_at^{h_a}x^a$. The condition $\frac{\partial f_{t}^{F}}{\partial t}(x_0,t_0)\neq 0$ in Proposition \ref{Jens} is then fulfilled when all polynomials $f_t^F(x)$ seen as polynomials in $n+1$ variables $(x,t)$ define non-singular hypersurfaces in $\R_{>0}^{n+1}$, which is the case for generic coefficients $c_a$.
\end{rem}

\subsection{Viro polynomials and critical systems}\label{Viro polynomials}

In the following, we fix an ordering \(\A = \{a_0,\ldots,a_{d+k}\}\) of the elements of $\A$.
Let $c\in (\R^{*})^{\A}$. Consider a function $h:\A\rightarrow \R, \ a\mapsto h_a$. For simplicity, we will often write \(c_i\) instead of \(c_{a_i}\) and $h_{i}$ instead of $h_{a_{i}}$. Consider the matrix 
\[\mAh=\begin{pmatrix}
1 & \cdots & 1\\
| & \cdots& |\\
a_0 & \cdots & a_{d+k}\\
| & \cdots& |\\
h_{0} & \cdots & h_{d+k}\\
\end{pmatrix}. \] 
We will call \(\mAh\) the {\em \(h\)-exponent matrix} associated to \(\A\) and \(h\).
It is the exponent matrix of 
\begin{equation}
	\label{Viro}
	f_{t}(x) =\sum_{a \in \A} c_a t^{h_a}x^{a}.
\end{equation}
seen as a (generalized) polynomial in $n+1$ variables $(x,t)$. The polynomial \(f_t \in \R[x]\) (where \(t\) is a parameter) is usually called Viro polynomial.
Consider the path $(f_{t})_{t \in ]0,+\infty[}$ in the space of polynomials with supports in $\A$. Recall that we identify $f_t$ with its coefficients
$c_t=\left(c_a t^{h_a}\right)_{a\in \A}$.
We are particularly interested by the polynomial $f$ given by $t=1$
with coefficients $c=(c_{a})_{a\in \A}$.
Note that $(c_t)_{t\in ]0,+\infty[}$ is a smooth path contained in one orthant of $(\R^*)^{\A}$.

For any face \(F\) of \(Q = \ch(\A)\), consider the truncation of $f_t$ to $F$:
$$f^{F}_{t}(x) =\sum_{a\in \AF} c_at^{h_a}x^a
.$$

Denote by $\AhF$ the set  of points $(a,h_{a})$ for $a \in \AF$ and by $\mAF$ (resp., $\mAhF$) the submatrix of $\mA$ (resp., $\mAh$) obtained by removing the columns which do not correspond to elements of $\AF$. 

The critical system associated to $f^F_{t} \in \R[t][x]$ is the system
\begin{equation}\label{Stilde_F}
 \ f_t^{F}=x_1\frac{\partial f_t^{F}}{\partial x_1} = \cdots =x_n\frac{\partial f_t^{F}}{\partial x_n}=0.\tag{\(\tilde{S}_F\)}
 \end{equation}
Thus $(x,t)$ is a positive solution of \eqref{Stilde_F} if and only if $x$ is a singular point of the hypersurface defined by $f_t^{F}$ in $(\R_{>0})^n$.

\begin{rem}
\label{translation}
Critical systems associated to a polynomial $f$ and its product $x^{a} \cdot f$ by a monomial have the same solutions in $\R_{>0}^{n}$ (more generally in $(\C^{*})^{n}$).
\end{rem}

System \eqref{Stilde_F} seen as a generalized polynomial system in $n+1$ variables $(x,t)$ has exponent matrix $\mAhF$.
It is worth noting that the set of positive solutions $(x,t)$ of~\eqref{Stilde_F} has positive dimension when $\dim F <n$. Let \(m = \dim(F)\) and send $\AF$ onto the \(m\)-coordinate subspace of \(\R^n\) given by the first $m$ coordinates using an affine transformation. This amounts to do a monomial change of coordinates so that \(f_t^F\) becomes a \(m\)-variate polynomial \(g_t^F\) whose exponent vectors
are the images by the previous affine transformation of the vectors in $\AF$ and the coefficients are unchanged (see Section \ref{monomial}).
Keeping the notation $(x_1,\ldots,x_n)$ for the new coordinates, we have $g^F_{t} \in \R[t][x_1,\ldots,x_m]$ and the corresponding critical system is
equivalent to
\begin{equation}\label{S_F}
	\ g_t^{F}=x_1\frac{\partial g_t^{F}}{\partial x_1} = \cdots =x_m\frac{\partial g_t^{F}}{\partial x_m}=0,\; \mbox{where $m=\dim F$.} \tag{\(S_F\)}
\end{equation}
Thus the set of positive solutions of \eqref{Stilde_F} is the product of the set of positive solutions
of \eqref{S_F} by $(\R_{>0})^{n-m}$.
\begin{rem}\label{same Gale dual}
 The set of positive solutions of the system~\eqref{S_F} does not depend on the choice of the previous affine transformation up to diffeomorphism, see Section \ref{monomial}.
 \end{rem}
 To summarize we have the following result.
 \begin{prop}
For any face $F$ of $Q$ we have $c_t \in {\overset{\frown}{\nabla}}_{\A_{F}}$ if and only if there exists $x \in (\R_{>0})^m$
such that $(x,t)$ is positive solution of \eqref{S_F}.
\end{prop}

We will also need in the following that adding the \(h\)-row increases the rank by \(1\). 
\begin{de}\label{de:h}
	The function \(h\) is said {\em compatible} if for every non pyramidal face \(F\) of \(Q\) we have \(\rk(\mAhF) = \rk(\mAF)+1\), or equivalently, $\codim \AhF=\codim \AF-1$.
\end{de}
Not being compatible, means that there exists a non pyramidal face \(F\) such that the vector \((h(a))_{a \in \AF}\) is a linear combination of the rows of \(\mAF\), that is to say, since $\AF$ has positive codimension, the vector $(h(a))_{a \in \A}$ belongs to some hyperplane of \(\R^{\A}\). Consequently, \(h\) is compatible as soon as it lies outside of the union of at most \(2^{\lvert \A \rvert}\) hyperplanes, which is satisfied for \(h\) generic enough.

\section{Bounds for positive solutions of critical systems}\label{Critical systems}
\subsection{Gale duality for critical systems}\label{Gale duality}

Under the assumption that $h$ is compatible, the codimension of the support $\A_{F}^{h}$ of the critical system \eqref{S_F} for the Viro polynomial $f_{t}^{F}$ is one less than the codimension of the support of $f^{F}$.
For instance, if the support of $f^{F}$ has codimension $2$, then the codimension of the support of the critical system \eqref{S_F} is $1$. Polynomial systems with support a set of codimension $1$ have been widely studied using Gale duality (for example~\cite{Bih07,BD,BDF21}). Gale duality for polynomial systems was introduced in \cite{BS07} (see also \cite{BS08}). Our main reference here will be \cite{BDG2}, Section 2.

Consider the critical system depending on \(n+1\) variables \((x,t)\):
\begin{equation}\label{S}
	\ f_{t}=x_1\frac{\partial f_{t}}{\partial x_1} = \cdots =x_n\frac{\partial f_{t}}{\partial x_n}=0\tag{\(S\)}
\end{equation}
where \(f_t\) is a Viro polynomial of support \(\A \subseteq \R^n\), \(n\) is the dimension of \(\A\) and the exponent matrix \(\mAh\) has rank \(n+2\). We will continue to denote the matrix associated to \(\A\) by \(\mA\) and the codimension of \(\A\) by \(k\).
We saw in Section \ref{Viro polynomials} that any critical system \eqref{S_F} is a particular case of \eqref{S} (taking $\A=\A_F$).
The coefficient matrix of~\eqref{S} is the matrix $\mC$ such that~\eqref{S} can be written as
$\mC \cdot (t^{h_0}x^{a_0},\ldots,t^{h_{n+k}}x^{a_{n+k}})^{T}=0$.
There is a basic necessary condition for~\eqref{S} to have a positive solution.
Given a solution $(x,t) \in \R_{>0}^{n+1}$ of~\eqref{S} 
the column matrix $v$ of entries $t^{h_{a}}x^{a} > 0$ belongs to the kernel of $\mC$. Let $\mD$ be a matrix Gale dual to $\mC$. Thus there exists $u \in \R^{k \times 1}$ such that $v=\mD\cdot u$, and we get that $\langle D_{i}, u \rangle >0$ for $i=0,\ldots,n+k$. Thus the row vectors $D_{i}$ of $\mD$ belong to an half space passing through the origin.
\begin{lem}\label{necessary}
If~\eqref{S} has a positive solution, then there exists a non-zero vector $y \in \R^{k}$ such that
$\langle D_{i}, y \rangle >0$ for $i=0,\ldots,n+k$, that is, the vectors $D_{i}$ belong to some open half space passing through the origin.
\end{lem}

A simple computation shows that the coefficient
matrix $\mC$ of~\eqref{S} is the matrix obtained by multiplying the $i$-th column of $\mA$ by $c_i$ for $i=0,\ldots,n+k$. As a consequence, a Gale dual matrix $\mD$ of $\mC$ is obtained from a Gale dual matrix $\mB$ of $\mA$ by dividing the $i$-th row of $\mB$ by $c_{i}$, that is,
\begin{equation}\label{D_{i}}
 D_{i}=\frac{1}{c_{i}} B_{i}, \quad i=0,\ldots,n+k.
\end{equation}

A first consequence is that if \(\A\) is pyramidal, then \(\mB\) will contain a zero row. So it will also be the case for \(\mD\). Then Lemma~\ref{necessary} directly implies that the critical system has no positive solutions, and we recover the well-know result:

\begin{cor}\label{cor:pyramidal}
	If \(\A\) is pyramidal, then the system~\eqref{S} has no positive solution.
\end{cor}

Let $\lambda \in \R^{(n+k+1) \times (k-1)}$ be any Gale dual matrix of $\mAh$.
Note that the kernel of $\mAh$ is contained in the kernel of $\mA$, thus any column of $\lambda$ is a linear combination of columns of $\mB$.

The \emph{Gale system} associated to~\eqref{S} is
defined by

\begin{equation}\label{G}
\prod_{i=0}^{n+k} {\langle  D_{i}, y \rangle}^{\lambda_{i,j}}=1, \quad j=1,\ldots,k-1.\tag{\(G\)}
\end{equation}

We might alternatively write $\frac{1}{c_{i}} \langle  B_{i}, y \rangle$ instead of $\langle  D_{i}, y \rangle$.
Note that \eqref{G} does not depend on the numbering of the elements of $\A$. Also, it can be noticed that, up to linear change of coordinates, the set of solutions of~\eqref{G} does not depend on the choice of the Gale dual matrices \(\mB\) and \(\mD\). Moreover, the equations of the system are homogeneous of degree zero since the columns of $\mB$ sum up to zero.

\begin{rem}
We already noticed that \eqref{S_F} is a particular case of \eqref{S}. Although \eqref{S_F} depends on a choice of an affine transformation,
Gale dual matrices of exponent and coefficients matrices are independent of this choice. This follows from the fact that applying an affine transformation to $\A_F$ amounts to multiply $A_F$ on the left by an invertible matrix, and that $A_F$ and the resulting matrix have obviously the same right kernels, see Section \ref{monomial}.
\end{rem}
Consider the positive cone generated by the rows of $\mD$
\begin{equation}\label{eq:poscone}
{\mathscr{C}}_{\mD}=\R_{>0}D_0+\cdots+\R_{>0}D_{n+k}.
\end{equation}
The dual cone of ${\mathscr{C}}_{\mD}$ is the cone
\begin{equation}\label{eq:Cnu}
{\mathscr C}_{\mD}^{\nu}=\{y \in \R^{k} \, : \, \langle D_i,y \rangle >0, i=0,\ldots, n+k\}.
\end{equation}
For any cone ${\mathscr{C}} \subset \R^n$ with apex the origin, its projectivization $\p{\mathscr{C}}$ is the quotient space ${\mathscr{C}} / {\simeq}$
under the equivalence relation $\simeq$ defined by: for all $y,y' \in {\mathscr{C}}$, 
we have  $y \simeq y'$ if and only if there exists $\alpha>0$
such that $y=\alpha y'$.

We saw that since $\mD$ is Gale dual to $\mC$, for any solution $(x,t)$ of~\eqref{S}, there exists a unique $y \in {\mathscr{C}}_{\mD}^{\nu}$ such that $t^{h_{i}}x^{a_{i}}=\langle D_{i},y \rangle$ for $i=0,\ldots,n+k$. 

Moreover, if \((x_0,t_0), (x_1,t_1)\) are two positive solutions associated to \(y_0\) and \(y_1\) in \({\mathscr{C}}_{\mD}^{\nu}\) such that \(y_0 \simeq y_1\), then there exists \(\alpha >0\) such that, for every \(a \in \A\) we have \(t_0^{h_a}x_0^a = \alpha t_1^{h_a}x_1^{a}\). But it implies that for all \(a\) in \(\A\), \((\ln t_0 -\ln t_1)h_a + \sum_{i=1}^n (\ln x_{0,i} - \ln x_{1,i})a_i = \ln \alpha\). Since \(\mAh\) is of maximal rank, we have that \(\alpha=1\) and \((x_0,t_0) = (x_1,t_1)\).

\begin{theo}[\cite{BDG2} Theorem 2.5 and~\cite{BS08} Theorem 2.1]
\label{Gale bijection}
There is a bijection preserving the multiplicities between the positive solutions of the system~\eqref{S} and the solutions of the Gale dual system~\eqref{G} in $\p{\mathscr{C}}_{\mD}^{\nu}$.
\end{theo}

We expand further the case $\codim \A=1$ since it will be useful later. Since
$\rk \mAh=\rk \mA+1$, the support $\mAh$ of the associated Viro critical system has codimension \(0\).
Most part of Example \ref{codimension 1} is known (see for instance  \cite{BHDR} and \cite{RR}).

\begin{ex}\label{codimension 1}
Assume $\codim \A=1$. Let $\lambda$ be any non zero vector in the kernel of $\mA$  and let  $c_{t}=(c_at^{h_{a}}) \in (\R^{*})^{\A}$. The corresponding critical system~\eqref{S} 
can be written $\mA \cdot (c_{a}t^{h_{a}}x^a)_{a}^{T}=0$, and $(x,t)$ is a positive solution of~\eqref{S} 
if and only if there exists $y \in \R$ such that 
\begin{equation}\label{Gale k=1}
c_{a}t^{h_{a}}x^{a}=\lambda_{a} y \text{ with } \frac{\lambda_{a}}{c_{a}} \cdot  y >0 \quad \text{for all } a \in \A.
\end{equation}
So we recover the bijective map of Theorem~\ref{Gale bijection}. It follows that~\eqref{S} 
has no positive solution when some $\lambda_{a}$ vanish, which is already known since in that case $\A$ is a pyramid and is thus defective (see also Corollary \ref{cor:pyramidal}). Assume that no $\lambda_{a}$ vanish, in other words, $\A$ is a circuit. It follows that if~\eqref{S} 
has a positive solution then all $c_{a} \lambda_{a}$ are either positive, or negative, which is equivalent to ${\mathscr{C}}_D^{\nu} \neq \emptyset$.
Moreover, if ${\mathscr{C}}_D^{\nu} \neq \emptyset$ then $\prod_{a \in \A} {(\frac{\lambda_{a}}{c_{a}t^{h_{a}}}\cdot y)}^{\lambda_{a}}=1$, which implies that
$\prod_{a \in \A} {(\frac{\lambda_{a}}{c_{a}t^{h_{a}}})}^{\lambda_{a}}=1$ since the coefficients $\lambda_{a}$ sum up to zero. Note that $\sum_{a \in \A} h_{a} \lambda_{a} \neq 0$ for otherwise $\lambda$ would belong to $\ker \mAh=\{0\}$. Thus $t$ is equal to the positive real number
\begin{equation}
\label{tcritical}
t_{\A}=\left(\prod_{a \in \A}
\left(\frac{\lambda_a}
{c_a}\right)^{\lambda_a} \right)^{\frac{1}{\sum_{a \in \A} h_{a} \lambda_{a}}}.
\end{equation} 
Then, fixing $t=t_{\A}$,
we see that any equality of~\eqref{Gale k=1} is a consequence of the others. Thus we can forget one equality of~\eqref{Gale k=1}, say the equality given by $a=a_{n+1}$. Choose another one, say the equality given by $a=a_{0}$, to get rid off $y$ and see that $x$ is a positive solution of a system of the form $x^{a_{i}-a_{0}}=d_i$, where $d_i$ are positive numbers for $i=1, \ldots, n$. The latter system is a system supported on a set of codimension zero, in other words it is a linear system up to a monomial change of coordinates. Thus it has a unique positive solution, and this solution has multiplicity one. It follows then from Theorem \ref{Gale bijection} that this solution is a double point of the hypersurface $\sum_{a \in \A} c_a {t}^{h_{a}}x^{a}=0$ when $ t=t_{\A}$. To resume, the hypersurface 
$\sum_{a \in \A} c_at^{h_{a}}x^{a}=0$ has no positive singular point if  the sign compatibility ${\mathscr{C}}_D^{\nu} \neq \emptyset$ is not satisfied or if $t \neq t_{\A}$. If the sign compatibility is satisfied, then the hypersurface given by $t=t_{\A}$ has only one singular point which is a double point.
\end{ex}

We now show that this behaviour continues to be true for supports of larger codimension when the coefficients are generic enough.

\begin{prop}
\label{simple}
Assume $\rk \mAh=\rk \mA+1=n+2$ and let $k=\codim \A$.  Then, all positive solutions of the critical system~\eqref{S} 
are simple solutions for generic enough coefficients $c_{a}$. Moreover, two distinct solutions have distinct coordinates $t$.
\end{prop}
\begin{proof}
We use Theorem \ref{Gale bijection}. 
If \(\A\) is defective, then the \(\A\)-discriminant variety is of codimension at least \(2\). Consequently, since the path \((c_at^{h_a})_{t\in]0,\infty[}\) has dimension \(1\), it avoids the discriminant variety
for generic enough coefficients $c_{a}$.
Moreover, we might assume that 
${\mathscr{C}}_D^{\nu} \neq \emptyset$ and that \(\A\) is not defective for otherwise the system~\eqref{S} 
has no positive solution at all. Then the Gale system~\eqref{G} consists of $k-1$ homogeneous equations (of degree \(0\))
\begin{equation*}
\phi_{j}(y)=1,  \quad j=1,\ldots,k-1,
\end{equation*}
where
$\phi_{j}(y)= \prod_{i=0}^{n+k} {\langle  D_{i}, y \rangle}^{\lambda_{i,j}}$,
$y=(y_{1},y_{2},\ldots,y_{k}) \in {\mathscr C}_D^{\nu}$, and $\lambda=(\lambda_{i,j})$ is a matrix Gale dual to $\mAh$. 

Let us consider \(J \in M_{k}(\R)\) defined for any \(1 \leq j \leq k\) by \(J_{k,j} = y_j\)  and \(J_{i,j} = \frac{\partial \phi_i}{\partial y_j}\) when \(1\leq i \leq k-1\). By Euler's homogeneous function Theorem, since each \(\phi_i\) is homogeneous of degree \(0\),  we know that \(\sum_{j=1}^k y_j \frac{\partial \phi_i}{\partial y_j} = 0\) for all \(i\). Consequently, the column vector \((y_1,\ldots,y_k)^T\) is in the kernel of the \(k-1\) first rows of \(J\) but not in the kernel of the last row.

In particular, a solution \(y \in {\mathscr{C}}_D^{\nu}\) of~\eqref{G} is a critical point if and only if the first \(k-1\) rows of \(J(y)\) are linearly dependent, which is equivalent to \(\det(J(y))=0\).

In~\cite{For19}, the author defined the cuspidal form \(P_{\A}\) of a finite set \(\A\):
\[
P_{\A}(y_1,\ldots,y_k) = \sum_{\substack{\sigma \subset [0,n+k] \\ \lvert \sigma \rvert =n}} \det(\hat{\mA}^{\sigma})^2 \prod_{\ell \in \sigma} \langle B_\ell,y\rangle
\]
where \(\hat{\mA}\) is the matrix \(\mA\) without its top row of \(1\)'s and \(\hat{\mA}^{\sigma}\) is the sub-matrix we get from \(\hat{\mA}\) by selecting the columns in \(\sigma\). In~\cite{For19} (Theorem 6.1), the author proved that this polynomial \(P_{\A}\) is identically zero if and only if \(\A\) is defective.

The proof of the following identity being mostly computational (succession of Laplace expansions and changes of variables), it is postponed to Appendix~\ref{app:PA}.
\begin{clm}\label{clm:PA}
	Assume the coefficients \(c_i\) are all non-zero. Then, for all \(y \in {{\mathscr{C}}_D^{\nu}}\), we have
	\[
	\det(J) = \gamma \left(\prod_{a \in \A} \langle B_a,y \rangle^{-1+\sum_{p=1}^{k-1}\lambda_{a,p}}\right) \cdot P_{\A}(y) \cdot \left( \sum_{j=1}^k y_i^2\right) \cdot \left( \sum_{a \in \A} h_a \langle B_a, y\rangle \right)
	\]
	where \(\gamma\) is a non-zero real constant which does not depend on \(y\).
\end{clm}

One can notice that the last factor can also be rewritten as a product of matrices: \((h_0,\ldots,h_{n+k}) \cdot B \cdot (y_1,\ldots,y_k)^T\). Since \(\rk(\mAh)=\rk(\mA)+1\), we know that the vector \(h\) is not in the left-kernel of \(\mB\), which means that the last factor is not identically zero. Thus by assumptions, \(\det(J)\) is not identically zero.

Note also that the vanishing of \(\det(J)\) does not depend on the coefficients $c_{a}$ (more precisely, they only appear in the factorized form in the constant \(\gamma\)). Writing $D_{i}=\frac{B_{i}}{c_i}$, we see that~\eqref{G} is equivalent to
\begin{equation}\label{equiv}
\psi_j(y) \defeq \prod_{i=0}^{n+k} {\langle  B_{i}, y \rangle}^{\lambda_{i,j}}
=\prod_{i=0}^{n+k}c_{i}^{\lambda_{i,j}}, \quad j=1,\ldots,k-1.
\end{equation}
The function \(\psi = (\psi_1,\ldots,\psi_{k-1}) \) does not depend on the $c_{a}$.

We know that \(\mathscr{C}_{\mD}^\nu\) is a non-empty open cone of apex the origin in \(\R^k\). So the set \(\{y \in \mathscr{C}_{\mD}^\nu \mid \det(J(y))=0\}\) is still a cone of apex \(0\) but of dimension at most \(k-1\). Since the functions \(\psi_j\) are homogeneous of degree \(0\), the image \(\psi(\{y \in \mathscr{C}_{\mD}^\nu \mid \det(J(y))=0\})\) is of dimension at most \(k-2\) in \(\R^{k-1}\).  Finally, since \((c_i)_i \mapsto (\prod_{i=0}^{n+k} c_i^{\lambda_{i,j} })_j \) is a submersion from \((\R^{*})^{n+k+1}\) to \(\R^{k-1}\) (because \(\rk \lambda =k-1\)), the set of polynomials \((c_a)\) which admit a point \(y \in \mathscr{C}_{\mD}^\nu\) verifying~\eqref{equiv} and vanishing \(\det(J)\) has codimension at least one. 
It follows then from Theorem~\ref{Gale bijection} that all positive solutions of the critical system~\eqref{S} are simple for generic enough coefficients $c_{a}$. 

Now let $y$ be a solution of \eqref{G} contained in ${\mathscr{C}}_D^{\nu}$. By Theorem \ref{Gale bijection} there exists a unique $(x,t) \in \R_{>0}^{n+1}$ such that
$$
t^{h_{i}}x^{a_{i}}=\langle  D_{i}, y \rangle, \; i=0,\ldots, n+k.
$$
Choose $u \in \ker \mA \setminus \ker \mAh$. Then, ${\sum_{i} h_{i} u_{i}} \neq 0$ and ${\sum_{i} a_{i} u_{i}} =0$ and thus $t$ is determined by $y$ via the equality

\begin{equation}\label{tcritk=k}
t^{\sum_{i} h_{i} u_{i}}=\prod_{i}\langle  D_{i}, y \rangle^{u_{i}}.
\end{equation}

Assume now that $y$ and $y'$ are two solutions of~\eqref{G} (or equivalently~\eqref{equiv}) contained in ${\mathscr{C}}_D^{\nu}$ such that the corresponding values $t$ and $t'$ are equal. Then, $y$ and $y'$ have the same image by the map sending $y \in {\mathbb{P}\mathscr{C}}_D^{\nu}$ to $(\prod_{i}\langle  B_{i}, y \rangle^{\mB_{i,j}})_{j=1,\ldots,k}$. But clearly the previous map is injective since $\mB$ has maximal rank and we conclude that \(y\simeq y^{\prime}\).

Consequently, if \((x,t)\) and \((x^{\prime},t)\) are two solutions of~\eqref{S}, then there are corresponding solutions \(y\) and \(y^{\prime}\) of the Gale dual system~\eqref{G}. But, it implies that \(y\simeq y^{\prime}\), and so \(x=x^{\prime}\).
\end{proof}

\subsection{Bounds in the
general case}\label{SS:Critical systems}
Consider the critical system~\eqref{S}. We present estimates on its number of positive solutions according to the dimension and the codimension of $\A$. We also present a necessary condition on the signs of coefficients $c_a$ for this number of positive solutions to be non-zero in any codimension and dimension. 

In the following, we will need to bound the number of positive values \(t\) such that the system~(\ref{S}) has a positive solution. This amounts therefore to bounding the number of positive solutions of a polynomial system with given support. This topic has been widely studied since Khovanski\v{\i}'s work~\cite{Khovanskii} on fewnomials. The approach to get the current best bound also goes via Gale duality. In~\cite{BS07}, the authors showed the following upper bound on the associated Gale system (see for example~\cite{Sottile}):
\begin{theo}\label{thm:BSGale}
	Let \(p_1(y),\ldots,p_{m+l}(y)\) be degree \(1\) polynomials on \(\R^l\) that, together with the constant \(1\), span the space of degree \(1\) polynomials. For any linearly independent vectors \(\{ \beta_1,\ldots,\beta_l\} \subset \R^{m+l}\), the number of solutions to 
	\[
		p(y)^{\beta_j} =1 \quad \text{ for } j=1,\ldots,l,
	\]
	in the positive chamber \(\{ y \in \R^l \mid \forall i\leq m+l,\ p_i(y) >0\}\) is less than
	\[	
		\frac{e^2+3}{4}2^{\binom{l}{2}}m^l.
	\]
\end{theo}

In our case, it will be sufficient to bound the number of positive solutions of the system~(\ref{G}) which is an equivalent system with \(m = n+1\) and \(l = k-1\).

\begin{cor}\label{boundSF}
	Assume $\rk \mAh=\rk \mA+1=n+2$ and let $k=\codim \A$. The number of positive solutions of a system~\eqref{S} is bounded by 
	\[	
	\frac{e^2+3}{4}2^{\binom{k-1}{2}}(n+1)^{k-1}.
	\]
\end{cor}

\begin{proof}
	By Theorem~\ref{Gale bijection}, this is enough to bound the number of solutions in \(\mathbb{P}\mathscr{C}_{\mD}^\nu\) of a system~\eqref{G}.
	
	The matrix \(\mD\) has rank \(k\), so there exist \(\alpha \in \A\) and \(\mG \in \mathrm{GL}_k(\R)\) such that the row \((\mD \cdot \mG)_{\alpha}\) is exactly the row vector \(e_k = (0,\ldots,0,1)\). Up to reordering the elements of \(\A\), assume that \(\alpha = a_0\). We consider the linear change of variables \(y^\prime = \mG^{-1} y\). The condition \(\langle D_a, y \rangle >0\) for all \(a\) becomes \(\langle D_a, \mG y^\prime \rangle >0\) for all \(a\). It implies in particular that \(D_{a_0}\cdot \mG\cdot  y^\prime = y^\prime_{k} >0\). So the number of solutions does not change by restricting the set to the affine chart given by \(y^\prime_{k} =1\).
	
	Consequently, \(y\) is a solution of~\eqref{G} in \(\mathbb{P}\mathscr{C}_{\mD}^\nu\) if and only if the element \((\tilde{y}_1,\ldots,\tilde{y}_{k-1}) \defeq (y_1^\prime/ y_k^\prime,\ldots,y_{k-1}^\prime/y_{k}^\prime) \in \R^{k-1}\) is a solution of 
	\begin{equation}\label{Gaffine}
		\prod_{i=1}^{n+k} p_i(\tilde{y})^{\lambda_{i,j}} =1, \quad j=1,\ldots,k-1
	\end{equation}
	where \(p_i(\tilde{y}) = \langle D_i \cdot \mG, (\tilde{y}_1,\ldots,\tilde{y}_{k-1},1)\rangle\), 
	which verifies \(p_i(\tilde{y}) >0\) for \(i=1,\ldots,n+k\) (notice that we removed one factor since \(p_0 \equiv 1\)).
	
	We can now apply the bound of Theorem~\ref{thm:BSGale} to get the desired result.
\end{proof}

We now present a necessary condition for the critical system to have positive solutions which depends on the signs of the coefficients $c_a$.
A subset \(J\) of \(\A\) is called
{\em coface} of \(\A\) if there exists a face $F$ of $Q$ such that $\A \setminus J = \AF$.

\begin{prop}\label{signs}
If there exists a non-empty coface $J$ of $\A$ such that $\{c_{a} \mid a \in J\} \subset \R_{>0}$
or $\{c_{a} \mid a \in J\} \subset \R_{<0}$ then ${\mathscr{C}}_D^{\nu}$ is empty and thus the critical system~\eqref{S} has no solution.
\end{prop}
\begin{proof}
Let $\ell: \R^n \rightarrow \R$ be some affine function. Then there is a row vector $R$ such that
$R \cdot A$ is the row vector with entries $\ell(a)$ for $a \in \A$.
Assume that ${\mathscr{C}}_D^{\nu}$ is non-empty. Then there exists a vector column \(v\)
with positive coordinates such that $C \cdot v=0$. But then  \(R \cdot C \cdot v=0\) so \(R \cdot C\) has to be the zero vector or to contain a positive and a negative entry. Recall that $C$ is the matrix obtained by multiplying the $i$-th column of $\mA$ by $c_i$ for $i=0,\ldots,n+k$. It follows that
the \(a\)-entry of \(R\cdot \mC\) equals $c_a \cdot \ell(a)$, and we get the desired result noting that if $J$ is a coface then there exists an affine function
$\ell: \R^n \rightarrow \R$ such that $\ell$ vanishes on the face $\AF=\A \setminus J$ and is positive on $J$.
\end{proof}

\begin{rem}\label{rem:cofaceGale}
    A direct consequence of the previous proof is the following well-known result (see for example Theorem~1 p.88 in~\cite{GS}): \(J\) is a coface of \(\A\) if and only if the origin belongs to the cone \(\sum_{a \in J} \R_{>0} B_a\) where \((B_a)_{a \in \A}\) are the rows of a Gale dual matrix \(\mB\) of \(\mA\). 

   Indeed, the fact that the origin belongs to the cone \(\sum_{a\in J} \R_{>0} B_a\) is equivalent to the existence of a row vector \(\mu \in \R_{\ge 0}^n\) in the left-kernel of \(\mB\) (and so in the row-span of \(\mA\)) such that \(\mu_a >0\) if and only if \(a \in J\). It means that there exists an affine function \(\ell\) such that \(\ell(a) > 0\) when \( a \in J\) and \(\ell(a)=0\) otherwise.
\end{rem}

\subsection{Bounds in small codimension}
\label{S:smallcodim}
We continue to assume that $\rk \mAh =\rk \mA+1$ which implies $\codim \Ah=\codim \A -1$.
Consequently, when $\codim \A =1$ (resp., $2$) the support $\mAh$ of the corresponding Viro critical system has codimension zero (resp., $1$). We focus on these two cases in this section. 

Start with the case where the codimension of \(\A\) equals $1$.
We say then that $c=(c_{a}) \in \R^{\A}$ (or the corresponding polynomial $f$) is \emph{sign compatible} with $\A$ if either $c_a \cdot \lambda_{a} >0$ for all $a \in \A$, or $c_a \cdot \lambda_{a} <0$ for all  $a \in \A$, where the $\lambda_{a}$'s are the coefficients in a given non zero
affine relation $\sum_{a \in \A} \lambda_{a} \cdot a=0$ on  $\A$. Note that this definition does not depend on the choice of such an affine relation.

The following result is well-known (see~\cite{GKZ},~\cite{BHDR}, and~\cite{RR}).

\begin{prop}\label{facecodim1}
Assume that $\A$ has codimension \(1\) and $\rk \mAh =\rk \mA+1$. If $c=(c_{a}) \in \R^{\A}$ is not sign compatible with $\A$ then~\eqref{S} has no positive solution. If $c$ is sign compatible with $\A$ then~\eqref{S} has exactly one positive solution, which is a simple solution.
\end{prop}
\begin{proof}
First, by Lemma~\ref{necessary}, we have that $(c_{a})$ is sign compatible with $\A$ if and only if the cone~\eqref{eq:Cnu}
is non-empty, which is a necessary condition for~\eqref{S} to have a positive solution. 

Since $\rk \mAh=\rk \mA+1$, we get that 
$\codim \Ah=0$ and thus, up to monomial change of coordinates, the system~\eqref{S} is equivalent to a linear system (with constant term). It is then easy to see that this linear system has one, and only one, positive solution (which is simple) precisely when $(c_{a})$ is sign compatible with $\A$, see Example \ref{codimension 1} for more details.
\end{proof}

\begin{rem}
If $\codim \A=1$ but $\A$ is not a circuit, then $\A$ is a pyramid and thus no polynomial \(f\) is sign compatible. Then~\eqref{S} has no positive solution (this fact could also be deduced from Corollary~\ref{cor:pyramidal}).
\end{rem}

Notice that in the case of codimension \(1\), the criterion given by Proposition~\ref{signs} becomes a characterization.

\begin{lem}\label{signs1}
	If $\codim \A=1$ then the necessary condition given in Proposition~\ref{signs} is equivalent to the sign compatibility of $f$ with $\A$ which is in turn equivalent to ${\mathscr{C}}_D^{\nu} \neq \emptyset$ .
\end{lem}

\begin{proof}
	Let \(\sum_{a \in \A} \lambda_a \cdot a = 0\) be any non-zero affine relation on the elements of \(\A\). Then the column matrix \(\mB=(\lambda_a)_a\) is a Gale dual matrix of \(\mA\). Then \(J \subset \A\) is a coface if and only if the origin belongs to the cone \(\sum_{a \in J} \R_{>0}\lambda_a\) (see Remark~\ref{rem:cofaceGale}), which means that \(\{\lambda_a \mid a\in J\}\) contains at least one strictly negative real number and at least one strictly positive one. Considering all cofaces with two elements gives then the result.
\end{proof}

We now turn to the codimension $2$ case. So assume $\codim \A=2$. 
We already know (Corollary~\ref{cor:pyramidal}) that if \(\A\) is pyramidal, then the critical system has no solution. So let us assume that \(\A\) is a non-pyramidal subset of $\R^n$.

We begin with a basic lemma.

\begin{lem}\label{colinear}
Let $\S$ be any non empty subset of \(\A\). 

\(\S\) is a flat of \(\MatrAd\) of rank \(1\) (i.e.\ all $B_{i}$ for $i \in \S$ are colinear and there does not exist $j \in \A \setminus \S$ such that $B_{j}$ is colinear to some (and thus any) $B_{i}$ with $i \in \S$) if and only if \(\A \setminus \S\) is a circuit.
\end{lem}
\begin{proof}
This is well-known (see for example~\cite{Oxley} Proposition 2.1.6) that \(\A\setminus \S\) is a circuit in \(\MatrA\) if and only if \(\S\) is a hyperplane of \(\MatrAd\) which is exactly a flat of rank \(1\) since \(\MatrAd\) has rank \(2\). 
\end{proof}

Consider the binary relation on the set \(\A\)
defined by "${i} \sim {j}$ if and only if the rows $B_{i}$ and $B_{j}$ of $B$ are colinear" (which is equivalent to \(i\) and \(j\) are parallel in \(\MatrAd\)).
This is an equivalence relation since $B_{i} \neq 0$ for all $i \in \A$. Note that this equivalence relation does not depend on the choice of the Gale dual matrix $B$ of $A$ (since it can be only defined from \(\MatrAd\)).

Denote by $\A/{\sim}$ the quotient space.

\begin{prop}\label{number of circuits}
Let \(\A\) be a non-pyramidal finite set in $\R^{n}$ of codimension \(2\).
The number $N$ of circuits ${\mathcal C} \subset \A$ such that $\dim{\mathcal C} < \dim \A$ is equal to the number of equivalence classes in $\A/{\sim}$ having at least two elements. As a consequence, we have $N+\lvert \A/{\sim}\rvert \leq \dim \A+3$.
\end{prop}
\begin{proof}
Since \(\A\) is not a pyramid we have \(\codim (\A\setminus\{a\})=1\) and \(\dim (\A\setminus\{a\})=\dim(\A)\)
for any \(a \in \A\). So if \(\mathcal{C}\subseteq \A\) has codimension \(1\), we have \(\dim(\mathcal{C}) < \dim(\A)\) if and only if \(\lvert \A \setminus \mathcal{C}\rvert \geq 2\). Then, the first part of the proposition is a consequence of Lemma~\ref{colinear}.
Moreover, we have \(N+\lvert \A /{\sim} \rvert \leq \lvert \A \rvert = \dim(\A)+3\).
\end{proof}

We now turn our attention to the critical system~\eqref{S} associated to $\A$ (we still have $\codim \A=2$).
Recall that $D_{i}=\frac{1}{c_{i}} B_{i}$ for all $i \in \A$, see \eqref{D_{i}}.
Thus vectors $D_{i}$ and $D_{j}$ are colinear if and only if $B_{i}$ and $B_{j}$ are colinear. Assume that the vectors $(D_i)_{i\in \A}$ are contained in an open half plane passing through the origin. Then, for any $i,j \in \A$, we have $i \sim j$ if and only if there is a \emph{positive} constant $c$ such that $D_i=c \cdot D_{j}$. Therefore, we can order the elements of $\A/{\sim}$ taking one representative $D_{u}$ for each equivalence class $u \in \A/{\sim}$ and declaring that for any $u,v \in \A/{\sim}$ we have $u < v$ if and only if the determinant $\det(D_{u}, D_{v})$ is (strictly) positive.

Let $u_{0} < u_{1}<\cdots<u_{s}$ be the elements of $\A/{\sim}$ ordered according to the previous ordering. Let $b$ any non zero vector in the kernel of $\mAh$.
For any equivalence class $u \in \A/{\sim}$, define $b_{u}=\sum_{i \in u}b_{i}$.

\begin{prop}\label{systeme principal}
If the vectors $D_{a}$ for $a \in \A$ are not contained in an open half plane passing through the origin, then~\eqref{S} has no positive solution. Otherwise, the system~\eqref{S} has at most
$$\signvar(b_{u_0}, b_{u_1},\ldots,b_{u_{s}}) \leq s$$ positive solutions.
\end{prop}
Here $\signvar(b_{u_0}, b_{u_1},\ldots,b_{u_{s}})$ is the number of sign changes between consecutive non zero terms in the sequence $(b_{u_0}, b_{u_1},\ldots,b_{u_{s}})$.

\begin{proof}
This is a direct consequence of~\cite{BD}, Theorem 2.9.
\end{proof}

\section{Viro hypersurfaces in the positive orthant}
\label{S:Viro}
Viro Patchworking Theorem has several versions, see \cite{Vir84}, \cite{Vir06}, \cite{Ris92}, \cite{GKZ}, \cite{Stu}, \cite{Bih} for instance, which work
for hypersurfaces or complete intersections defined by Laurent polynomials and contained in either the positive orthant, the real torus or the real part of a projective toric variety. When the height function is generic enough, only the height function and the signs of the coefficients (not their absolute values) are needed and in this case we speak about Combinatorial Viro Patchworking Theorem.

Here we use the simplest version: the case of an hypersurface in the positive orthant, but with a slight difference in that we consider hypersurfaces defined by generalized polynomials (polynomials with real exponent vectors) instead of Laurent polynomials. Clearly the proof of the classical Viro patchworking Theorem for hypersurfaces in the positive orthant also works for hypersurfaces defined by generalized polynomials (see the proofs contained in \cite{Ris92}, \cite{GKZ}, \cite{Stu}, \cite{Bih} for instance).
\subsection{Viro Patchworking Theorem}
\label{SS:Viro}

Let $\A \subset \R^n$ be a finite set.
Consider a Viro polynomial 
$$
f_t=\sum_{a\in \A}c_a t^{h_a}x^{a}$$
where $c=(c_a)_{a \in \A}$, h=$(h_a)_{a \in \A}$ are collections of real numbers and $t$ is a real positive parameter.
We assume here that $c_a \neq 0$ for all $a \in \A$.
Denote by $Q$ the convex hull of $\A$ and by $\hat{Q}$ the convex hull of \(\Ah=\{(a,h_a) \in \R^{n+1} \mid a \in \A\}\).
Assume that $\hat{Q}$ has full dimension $n+1$. A lower facet of $\hat{Q}$ is a facet with outward normal vector with negative last coordinate.
Projecting lower facets of $\hat Q$ (together with their faces) onto \(\R^n\) by the projection forgetting the last coordinate,
we get a polyhedral subdivision $\tau$ of $Q$. Each polytope $F$ of $\tau$ is the image of a face $\hat{F}$ of a lower facet
of $\hat Q$ by the previous projection. Consider the polynomial $f^F$ obtained by setting $t=1$ in the polynomial
$$
f_t^F=\sum_{(a,h_a)\in \Ah \cap \hat{F}}c_a t^{h_a}x^{a}.$$
Note that the support of $f^F$ is contained in $\A \cap F$ but the inclusion is strict if there exists $a \in \A \cap F$ such that
$(a,h_a) \notin \Ah \cap \hat{F}$.

\begin{assump}
\label{A:gen}
For each $F \in \tau$, the hypersurface $V_{>0}(f^F) \subset \R_{>0}^n$ is smooth.
\end{assump}

Consider for each $F \in \tau$ the chart $C_{F}^{\circ}(f^{F}) \subset \mbox{relint}(F)$. It is homeomorphic to 
the hypersurface defined by $f^{F}$ in  $X_{\A_{F}}^{\circ}(\R_{>0})$, see Subsection \ref{S:Adiscr}.
Let $Z \subset Q$ denote the union of all charts $C_{F}^{\circ}(f^{F})$ for $F \in \tau$.

Consider the following assumption which is satisfied for generic enough height function $h=(h_a)_{a \in \A}$ (and any coefficients $(c_a)_{a \in \A}$) and
which we are going to show is stronger than Assumption \ref{A:gen}.

\begin{assump}
\label{A:comb}
Each lower facet $\hat F$ of $\hat Q$ is a simplex which intersects
$\Ah$ at its set of vertices.
\end{assump}

Assume that Assumption \ref{A:comb} is satisfied. Then the polyhedral subdivision $\tau$ is a triangulation and for each simplex $F \in \tau$ the support of the polynomial $f^{F}$ coincides with the set of vertices of $F$. Since the $\A$-discriminant variety $\nabla_{\A}$ is empty when $\A$ is a simplex, Assumption \ref{A:gen} is then automatically satisfied. Moreover, each chart $C_{F}^{\circ}(f^{F})$ is homeomorphic to a hyperplane section of the relative interior
of $F$ which separates vertices $a$ with positive coefficients $c_a$ from those with negative coefficients $c_a$. To see this, one can multiply $f^F$ by a monomial and use a monomial change of coordinates with real exponents, which corresponds to taking the image of the set of vertices of $F$ by an affine transformation of $\R^n$, to get a Laurent polynomial of degree one, for which the result is obvious. It follows that the hypersurface $Z \subset Q$ is homeomorphic to the piecewise linear hypersurface $L \subset Q$ obtained by the following so-called combinatorial patchworking contruction or T-construction.

Let $\A_{\tau} \subset \A$ be the set of vertices of $\tau$.
To each $a \in \A_{\tau}$, we associate the sign of the coefficient $c_{a}$.
If a \(n\)-dimensional simplex \(\delta\) of \(\tau\) has vertices of different signs, consider the edges from \(\delta\) which have endpoints of opposite signs, and take the convex hull of the middle points of these edges. Let us denote by \(L\) the union of the taken hyperplane pieces. This is a piecewise linear hypersurface contained in \(Q\).

\begin{prop}
\label{P:Patchwork}
Under Assumption \ref{A:gen} there exists $t_0>0$ such that for any real number $t$ with $0<t<t_0$ there exists an homeomorphism $(\R_{>0})^n \rightarrow \mbox{Int}(Q)$ sending $Z_{>0}(f_t)$ to $\mbox{Int}(Q) \cap Z$. If furthermore Assumption \ref{A:comb} is satisfied, then for any $0<t<t_0$ there exists an homeomorphism as before sending $Z_{>0}(f_t)$ to $\mbox{Int}(Q) \cap L$.
\end{prop}

\subsection{Bounds for the number of connected components of Viro hypersurfaces}
We present two bounds for the number of connected components of a hypersurface of the positive orthant defined by a Viro polynomial $f_t$ for $t>0$ small enough.
The first one deals with Viro polynomials with any number of variables and any number of monomials but satisfying Assumption
\ref{A:comb} (combinatorial patchworking construction). For the second bound, it is assumed that the codimension of $\A$ equals two and that the coefficients $c_a$ are generic enough so that Assumption \ref{A:gen} is satisfied.

The following result is elementary. Consider the piecewise linear hypersurface $L$ constructed above.

\begin{lem}
\label{L:basic}

The following properties of \(L\) are satisfied:

	\begin{enumerate}
		\item  \(b_0(V_{>0}(f_t)) = b_0(L)\) for $t>0$ small enough.
		\item Each connected component of \(Q \setminus L\) (we will call them \emph{chambers} in the following) contains at least one vertex of \(\A\).
		\item For any connected component \(C\) of \(L\), the set \(Q\setminus C\) has two connected components.
		\item Each connected component \(C\) of \(L\) has in its neighboring exactly two chambers (we will call them, its {\em neighboring chambers}). Furthermore, any path from one point of a connected component of \(Q\setminus C\) to a point in the other connected component intersects \(C\) and so also its two neighboring chambers.
	\end{enumerate}
\end{lem}
\begin{proof}
 Let \(\delta\) be a \(p\)-dimensional simplex of the triangulation \(\tau\). Let \(\delta_+\) be the convex hull of \(\{a \in \A_{\tau}\cap\delta \mid c_a>0\}\) and \(\delta_-\) be the convex hull of \(\{a\in \A_{\tau}\cap\delta \mid c_a<0\}\). As \(\delta\) is a simplex, one can consider the barycentric coordinates of each point of $\delta$ with respect to the vertices of \(\delta\): 
\begin{align*}
	x \in \delta & \iff \exists ! \gamma \in [0,1]^{\A_{\tau}\cap \delta} \begin{cases}
		x = \sum_{a \in (\A_{\tau}\cap \delta)} \gamma_a \cdot a \\ \sum_{a\in (\A_{\tau}\cap\delta)} \gamma_a =1.
	\end{cases}
\end{align*}
We define a function $\mu: Q \rightarrow \R$ as follows. For any $x \in Q$, consider a simplex $\delta \in \tau$ containing $x$ and set \(\mu(x)=2\sum_{a\in \delta_+} \gamma_a(x)-1\) where \(\gamma\) are the barycentric coordinates of \(x\) with respect to the vertices of $\delta$ ($\mu(x)$ does not depend on the chosen simplex $\delta$).
Over each $\delta \in \tau$, the function \(\mu\) is affine and verifies $\mu(\delta_+)=\{1\}$, $\mu(\delta_{-})=\{-1\}$.

Notice that \(L = \mu^{-1}(\{0\}) \subset Q\). Any simplex of \(\tau\) is contained in a \(n\)-dimensional simplex of \(\tau\). Consequently, \(L\) is determined by the \(n\)-simplices of $\tau$: \(L = \bigcup_{\delta:  n-\text{simplex}} \mu_{| \delta}^{-1}(\{0\})\).

Proposition \ref{P:Patchwork}
implies that 
as soon as \(t\) is positive and either small enough or large enough, there exists a homeomorphism of pairs between \((\R_{>0}^{n},V_{>0}(f_t))\) and \((\mathring{Q},L\cap \mathring{Q})\), where $\mathring{Q}$ is the interior of $Q$. So \(b_0(V_{>0}(f_t)) = b_0(L \cap \mathring{Q}) \) and it will be sufficient to bound from above $b_0(L \cap \mathring{Q})$.

Furthermore, the set \(\{x \in Q \mid \lvert \mu(x) \rvert <1\}\) can be interpreted as a trivial fibration over \(L\). Indeed, if \(x\in Q\) satisfies \(\lvert \mu(x)\rvert <1\), then there exists also a unique pair \((a_-,a_+) \in \delta_-\times\delta_+\) such that \(x = \frac{(\mu(x)+1){a_+} + (1-\mu(x)){a_-}}{2}\) (notice that the pair does not depend on the choice of the simplex \(\delta\)). Consequently, the function \(L_\mu\) which sends \(x\) to \(((a_-+a_+)/2,\mu(x))\) is a homeomorphism between  \(\{x \in Q \mid \lvert\mu(x)\rvert <1\}\) and \(L\times \left]-1,1\right[\).

Assertions of the lemma follow from this interpretation.
\begin{itemize}
	\item First, by Proposition~\ref{P:Patchwork}, it is enough to prove that \(b_0(L) = b_0(L\cap \mathring{Q})\). Since,  \(L = \bigcup_{\delta:  n-\text{simplex}} \mu_{| \delta}^{-1}(\{0\})\), we have that every connected component of \(L\) contains at least one connected component of \(L \cap \mathring{Q}\). Let us show that it contains only one. Let \(x\) be in the closure of two connected components \(C\) and \(C'\) of \(L\cap \mathring{Q}\). Hence \(x\in L\). It means that \(x\) belongs to the interior of a face \(F\) of \(Q\). Let \(y\in C\) and \(y' \in C'\) be points close to \(x\). By convexity of \(Q\), the segment \([y,y']\) lies in the interior of \(Q\). If \(y\) and \(y'\) are close enough to \(x\), \(F\) is a face of every \(n\)-simplex of \(\tau\) intersected by \([y,y']\). This implies that there exists a path in \(L\cap \mathring{Q}\) from \(y\) to \(y'\), proving that \(b_0(L) = b_0(L\cap \mathring{Q})\).
	\item Second, we have that \(Q\) is partitioned into \((\mu^{-1}(\R_{<0}),\mu^{-1}(\R_{>0}),L)\). So, by construction of \(\mu\), any \(x \in Q\setminus L\) with \(\mu(x) \neq 0\) is path-connected to a vertex of \(\A\) inside \(\mu^{-1}(\mu(x)\cdot\R_{>0})\). It is the second point of the lemma.
	\item Third, each connected component \(C\) of \(L\) is homeomorphic to a smooth component of \(V_{>0}(f_t)\) by Viro pairs homeomorphism which implies the third assertion.
	\item Fourth, to each connected component \(C\) of \(L\), we can associate the following two sets \(C_- \defeq L_{\mu}^{-1}(C\times\left]-1,0\right[)\) and \(C_+ \defeq L_{\mu}^{-1}(C\times\left]0,1\right[)\) which are homeomorphic to \(C\times \left]0,1\right[\), and so, each one is connected. Thus, each connected component \(C\) of \(L\) has in its neighboring exactly two chambers -- the one containing \(C_-\) and the one containing \(C_+\). 
	This completes the proof.
\end{itemize}
\end{proof}

We are now ready to bound \(b_0(V_{>0}(f_t)) \) for $t>0$ small enough when the height function $h$ is generic enough.

\begin{theo}\label{patchwork}
	Assume that Assumption \ref{A:comb} is satisfied.
	Then, there exists $0<t_0<1$ such that for every $t\in \left]0,t_0\right]$, we have
	\(b_0(V_{>0}(f_t)) \leq k+1\). 
	If \(n \geq 2\) and \(k \geq 2\), then we have the better bound \(b_0(V_{>0}(f_t)) \leq k\).
\end{theo}

\begin{proof}
By Proposition \ref{P:Patchwork} and Lemma~\ref{L:basic}, it is enough to show that the given bound holds for \(b_0(L)\).
	
	Let us construct the dual graph \(G\). The vertices of \(G\) are the chambers of \(Q \setminus L\). There is an edge between two chambers if their closures
	have a non-empty intersection (equivalently, they are the neighboring chambers of a same connected component of \(L\)). We claim that \(G\) is in fact a tree (i.e., it does not contain cycles). Indeed, if \(e=(u,v)\) is an edge from \(G\), it means that \(u\) and \(v\) are two neighbouring chambers of a same connected component \(C_e\) of \(L\). Any path from \(u\) to \(v\) in $G$ corresponds to a path in \(Q\) between a point of \(u\) to a point of \(v\). We saw that such a path has to intersect \(C_e\). That is to say, any path in \(G\) from \(u\) to \(v\) contains the edge \(e\). Consequently, \(G\) is a tree, and so its number of edges equals its number of vertices minus one. So we already get that \(b_0(L) \leq n+k\) since there is at least a point of \(\A\) in each chamber. In particular we get the stated bound in the case \(n=1\).
	
	From now, assume \(n \geq 2\). To get the stated bound, we want to ensure that some chambers contain several points of \(\A\). Let us see the triangulation $\tau$ as a pure\footnote{A \(n\)-complex is called \emph{pure} if any simplex is a face of a \(n\)-dimensional simplex.} simplicial \(n\)-complex 
and let $\mathcal{K}$ be any pure $n$-dimensional subcomplex of $\tau$.
Let \(\kappa = \lvert \{\text{vertices in }\mathcal{K}\}\rvert -n-1 \geq 0\). We show by induction on \(\kappa\) that \(b_0(L \cap \mathcal{K}) \leq \kappa + \mathbb{1}_{\kappa<2}\) which would prove the proposition taking $\mathcal{K}=\tau$.
	
	Assume that \(\kappa=0\), i.e., \(\mathcal{K}\) has \(n+1\) vertices. Since \(\mathcal{K}\) is pure of dimension \(n\), then \(\mathcal{K}\) contains an unique $n$-simplex. Consequently, \(b_0(L\cap \mathcal{K}) \leq 1\) which is the required bound.
	
	Assume now that \(\kappa=1\), i.e., \(\mathcal{K}\) has \(n+2\) vertices.  Let \(u\) be one of the leaves of \(G\). So \(u\) is connected to the remainder of \(G\) by an edge \(e=(u,v)\) corresponding to a connected component \(C_e\) of \(L\). Let \(\mathcal{K}^\prime\) be the pure subcomplex of  \(\mathcal{K}\) consisting of all $n$-simplices (together with their faces) of $\mathcal K$ without vertices in the chamber $u$. Note that \({\mathcal K'} \cap C_{e} = \emptyset\) and that \(\mathcal{K}^\prime\) is either empty or a pure $n$-dimensional complex.
Let $C'$ be a connected component of $L$ such that $C' \neq C_e$. Since $u$ is a leaf of $G$, we get that $C'$ does not intersect any $n$-simplex of $\mathcal K$ having a vertex in $u$, in other words, $C'$ is contained in \(\mathcal{K}^\prime\).
So \(b_0(L\cap \mathcal{K}^\prime) = b_0(L\cap\mathcal{K})-1\). If \(\mathcal{K}^\prime\) is empty, then \(b_0(L\cap \mathcal{K}) \leq 1\) which is what is wanted. Otherwise, \(\mathcal{K}^\prime\) is pure $n$-dimensional and thus contains at least \(n+1\) vertices, and so exactly $n+1$. By induction, \(b_0(L \cap \mathcal{K}^\prime) \leq 1\), which implies that \(b_0(L \cap \mathcal{K}) \leq 2\) which is the required bound.
	
	Before going to the following case, let us focus on the case where the bound \(2\) is reached. In this case \(G\) is a path of length \(2\): \(u-v-w\). As said before, \(u \cup v\) and \(v\cup w\) contain exactly $n+1$ vertices, thus  \(u\) and \(w\) contain each exactly one vertex, and  \(v\) contains \(n\) vertices spanning an hyperplane of \(\R^n\). 
	
	Assume now that \(\kappa=2\), i.e., \(\mathcal{K}\) has \(n+3\) vertices. We begin as in the case $\kappa=1$. Let \(z\) be a leaf of \(G\). Let \(\mathcal{K}^\prime\) be the pure subcomplex of  \(\mathcal{K}\) consisting of all $n$-simplices (together with their faces) of $\mathcal K$ without vertices in the chamber $z$. Then as before, we get \(b_0(L\cap \mathcal{K}^\prime) = b_0(L\cap\mathcal{K})-1\). Thus, by induction \(b_0(L \cap \mathcal{K}^\prime) \leq 2\), and so  \(b_0(L \cap \mathcal{K}) \leq 3\). Assume now that \(b_{0}(L \cap \mathcal{K})=3\), which implies  \(b_0(L \cap \mathcal{K}^\prime) = 2\). We know that \(z\) is a leaf and from the case $\kappa=1$ above (using \(b_0(L \cap \mathcal{K}^\prime) = 2\)), we get that the remainder of the graph is a length-\(2\) path \(u-v-w\) such that $u,w$ contains each one vertex and $v$ contains $n$ vertices. Furthermore, since \(\mathcal{K}^\prime\) contains \(n+2\) vertices, the chamber \(z\) contains only one vertex. The edge leaving \(z\) corresponds to a connected component $C_{z}$ of \(L\), so the union of the two neighbouring chambers
($z$ and one chamber among $u,v,w$) of $C_{z}$ contains at least \(n+1\) vertices (the vertices of a $n$-simplex of  \(\mathcal{K} \) intersected by $C_{z}$), and thus it is $v$.
Consequently, the graph \(G\) is a star of center \(v\) with three branches \(u\),\(w\), and \(z\).
This implies that the convex hulls of \(\A \cap (v\cup u)\), \(\A \cap (v\cup w)\), and \(\A \cap (v\cup z)\) are three $n$-simplices of $\mathcal K$ with disjoint interiors and sharing a common \((n-1)\)-dimensional face, which is impossible. Consequently, \(b_0(L \cap \mathcal{K}) \leq 2\).

To finish, assume now that \(\kappa >2\). Let \(u\) be a leaf of \(G\). We construct again a pure \(n\)-complex \(\mathcal{K}^\prime\) by removing vertices from \(u\) (which has at most \(n+\kappa\) vertices). By induction, \(b_0(L \cap \mathcal{K}^\prime) \leq \kappa-1\), and so  \(b_0(L \cap \mathcal{K}) \leq \kappa\). 
\end{proof}

The bound of Theorem \ref{patchwork} is sharp for \(n=1\): it suffices to take coefficients $c_a$ such that consecutive coefficients have opposite signs and an height function such that
$\A$ equals the set of vertices of the corresponding triangulation.

The next example shows that the bound continues to be sharp in dimension \(n \geq 2\) (up to the multiplicative factor \(\frac{1}{1+1/n!}\)).

\begin{ex}
Let us fix \(n \ge 2\). We define a family of piecewise linear hypersurfaces \((L_p)_{p \in \mathbb{N}}\) of dimension \(n-1\). Each hypersurface \(L_p\) is obtained by the combinatorial patchworking construction from a triangulation with \(n+k+1\) vertices of an \(n\)-simplex. The construction is designed so that the ratio between \(k\) and the number of connected components of \(L_p\) goes to \(\frac{1}{1+1/n!}\) when \(p\) goes to infinity.
        This construction is given by the \(p\)-fold edgewise subdivision of the simplex as defined in~\cite{EG99,BR05} (see also~\cite{KKMS73} Example 2.3). Let \(p\) be a positive number. Given an \(n\)-simplex \(\sigma\) with vertices \(P_0,\ldots,P_n\), let \(\tau_p\) be the \(p\)-fold edgewise subdivision of \(\sigma\). Its vertices are all points with barycentric coordinates \((b_0,\ldots,b_n)\) (with respect to \(P_0,\ldots,P_n\)) where each \(b_i\) is an integer multiple of \(1/p\). So, each vertex is written as \(\sum_{i=0}^n \frac{a_i}{p}P_i\) where the \(a_i\) are non-negative integers such that \(\sum_{i=0}^n a_i = p\). Consequently, the triangulation has \(\binom{p+n}{n}\) vertices. It is known (see again~\cite{EG99}) that the number of \(n\)-simplices is exactly \(p^n\). Moreover, this triangulation is regular (Lemma 2.4 in~\cite{KKMS73}).
    We refine \(\tau_p\) into a triangulation \(\tau'_p\) by adding a new vertex in the interior of each \(n\)-simplex of \(\tau_p\). The new triangulation \(\tau'_p\) is still regular. We associate to each vertex of \(\tau_p\) a positive sign, and to each other vertex of \(\tau'_p\) a negative one. It defines the piecewise linear hypersurface \(L_p\). Notice that \(L_p\) has as many connected components as there are \(n\)-simplices in \(\tau_p\). The codimension \(k\) is given by the total number of vertices of \(\tau'_p\) minus \(n+1\) which is \(\binom{p+n}{n}+p^n-n-1 = (1+1/n!)p^n + O(p^{n-1})\).
\end{ex}

We show now that we get a similar result by relaxing the constraint over \(h\) when the codimension is small (we are not anymore in the setting of the combinatorial Viro patchworking Theorem).

\begin{prop}\label{non patchwork}
	Let $\A$ be a finite set in $\R^{n}$ such that $\codim \A=2$ and $\dim \A \geq 2$.
	Consider a Viro polynomial $f_{t}(x)=\sum_{a \in \A} c_a t^{h_a}x^{a}$. Assume that
	$\rk \mAh =\rk \mA+1$. Then for generic enough coefficients $c_{i}$, there exists $t_{0}>0$ such that for any $0<t<t_0$ we have
	$b_{0}(V_{>0}(f_{t})) \leq 2$.
\end{prop}

\begin{proof}
Denote by $\tau$ the convex polyhedral subdivision of $Q=\ch(\A)$ obtained by projecting the lower facets of the convex hull of \(\{(a,h_a) \in \R^{n+1} \mid a \in \A\}\). Perturbing slightly the coefficients $c_a$ if necessary, we may assume that Assumption \ref{A:gen} is satisfied and apply Proposition \ref{P:Patchwork} to get an homeomorphism between the hypersurface $Z$ defined in Subsection \ref{S:Viro} and $V_{>0}(f_t)$ for $t>0$ small enough. We now show that the hypersurface $Z$ is homeomorphic to some piecewise linear hypersurface $L$ obtained by the combinatorial patchworking construction starting from $\A$ and some associated coefficient and height functions as in Subsection \ref{S:Viro}. Then the result will follow directly from Theorem~\ref{patchwork}.

Let $\A'$ be the basis of $\A$ (recall that $\A'=\A$ if $\A$ is not a pyramid). Then $Q'=\ch(\A')$ is a face of $Q$ and the polytopes of $\tau$ contained in $Q'$ form a polyhedral subdivision of $Q'$. Moreover all polytopes of $\tau$ which do not belong to $\tau'$ are pyramids over polytopes of $\tau'$. From $\rk \mAh =\rk \mA+1$ we get that for any polytope $P' \in \tau'$ we have $P' \cap \A' \neq \A'$ which implies $\codim P' \cap \A' \leq 1$ since $\codim \A'=2$ and $\A'$ is not a pyramid (Lemma \ref{codim-pyramid}).
If $\codim P' \cap \A' =0$ then $P'$ is a simplex with set of vertices $P' \cap \A'$, and vice versa. Consequently, 
we have $\codim P' \cap \A' =0$ for any $P' \in \tau'$ if and only if $\tau'$ is a triangulation of $\A'$ with set of vertices $\A'$, which in turn is equivalent to the fact that $\tau$ is a triangulation with set of vertices $\A$ and that $P \cap \A$ equals the set of vertices of $P$ for any $P \in \tau$. In that case $Z$ is homeomorphic to a piecewise linear hypersurface $L$ as required and the result follows from Theorem~\ref{patchwork}.

Assume now that there exists $P' \in \tau'$ such that $\codim P' \cap \A' =1$. Then  $P' \cap \A' $ is a circuit or a pyramid over a circuit (its basis), and this circuit is equal to $F' \cap \A'$ for some face $F'$ of $P'$. In particular, $F'$ belongs to the subdivision $\tau'$. Consider the star $st_{\tau'}(F')$ of $F'$ in $\tau'$, that is, the set of all polytopes in $\tau'$ having $F'$ as a face. For any $\Delta' \in st_{\tau'}(F')$, we have $\codim \Delta' \cap \A' =\codim F' \cap \A' =1$ and thus $\Delta '\cap \A'$ is a pyramid over $F' \cap \A'$. Coming back to the subdivision $\tau$ of $Q$, we get that for all polytopes $\Delta$ in $st_{\tau}(F')$ (polytopes of $\tau$ having $F'$ as a face)  the set $\Delta \cap \A$ is a pyramid over $F' \cap \A' $. 

In passing we note that a consequence of the previous fact is that there exists at most one $F' \in \tau$ such that  $F' \cap \A$ is a circuit. Indeed, assume on the contrary that there are two distinct $F'_{1}, F'_{2} \in \tau$ such that $F_{i}' \cap \A$ is a circuit for $i=1,2$. Then, for any $n$-polytopes $\Delta_{1} \in st_{\tau}(F_{1}')$ and $\Delta_{2} \in st_{\tau}(F_{2}')$ we have $\Delta_{1} \neq \Delta_{2}$ since a pyramid over a circuit cannot contain two distinct circuits. Thus $\Delta_{1} \cap \Delta_{2}$ is a possibly empty common face of dimension strictly smaller than $n$. Moreover, we get $\codim(\Delta_{1} \cap \Delta_{2} \cap \A)=0$ which yields then $|\Delta_{1} \cap \Delta_{2} \cap \A| \leq n$. Thus, $|(\Delta_{1} \cup \Delta_{2}) \cap \A|= 2(n+2)-|\Delta_{1} \cap \Delta_{2} \cap \A| \geq n+4 >n+3=|\A|$ giving a contradiction. 

Consider now any height function $g: F' \cap \A' \rightarrow \R$ which is not the restriction of an affine function and extend it by $0$ on the other points of $\Delta \cap \A$ for all polytopes $\Delta \in st_{\tau}(F')$. 
For each $\Delta \in st_{\tau}(F')$, consider the Viro polynomial $f_{|\Delta,t}(x)=\sum_{a \in \Delta \cap \A} c_{a}t^{g_a}x^{a}$. For $t=1$ this gives the polynomial $f_{|\Delta}$ which defines a nonsingular hypersurface in $(\R_{>0})^{n}$ by assumption. 
The only face $F$ of $\Delta$ such that $F \cap \A$ is not a pyramid is $F'$, so only the facial critical system corresponding to $F'$ can have a positive solution (see Corollary~\ref{cor:pyramidal}). This provides at most one value $t_0$ of $t$ such that the isotopy type of $\{f_{|\Delta,t}=0\} \subset (\R_{>0})^{n}$ might change passing through $t_{0}$. The key point is that $t_{0}$ is determined by $\sum_{a \in F' \cap \A'} c_{a}t^{g_a}x^{a}$ and thus does not depend on $\Delta \in st_{\tau}(F')$.
It follows that either for all $\Delta \in st_{\tau}(F')$ the hypersurface $\{f_{|\Delta}(x)=0\} \subset (\R_{>0})^{n}$ is isotopic to the hypersurface $\{f_{|\Delta,t}(x)=0\} \subset (\R_{>0})^{n}$ for any $0<t<t_{0}$, or for all $\Delta \in st_{\tau}(F')$ the hypersurface $\{f_{|\Delta}(x)=0\} \subset (\R_{>0})^{n}$ is isotopic to the hypersurface $\{f_{|\Delta,t}(x)=0\} \subset (\R_{>0})^{n}$ for any $t>t_{0}$ (both are true if the critical system corresponding to $F'$ has no positive solution).
For any $\Delta \in st_{\tau}(F')$, denote by $\tau_{\ell}(\Delta)$ (resp., $\tau_{u}(\Delta)$) the convex polyhedral subdivision of $\Delta$ obtained by projecting the lower faces (respectively, the upper faces) of the convex hull of \(\{(a,g_a) \in \R^{n+1} \mid a \in \Delta \cap \A\}\).
Since $g$ is not the restriction of an affine function on $F'$, the subdivisions $\tau_{\ell}(F')$ and $\tau_{u}(F')$ are non-trivial subdivisions, and are thus triangulations. Moreover, for each $\Delta \in st_{\tau}(F')$ the subdivision $\tau_{\ell}(\Delta)$ (resp., $\tau_{u}(\Delta)$) is obtained by subdividing $\Delta$ along $\tau_{\ell}(F')$ (resp., $\tau_{u}(F')$). Since $\Delta \cap \A$ is a pyramid over $F' \cap \A'$, it follows that  $\tau_{\ell}(\Delta)$ (resp., $\tau_{u}(\Delta)$) is a triangulation. It follows that the result of taking the union of the charts
$C_{\Delta}^{\circ}(f^{\Delta})$ for $\Delta \in st_{\tau}(F')$ can be obtained up to homeomorphism via the combinatorial patchworking process using either the triangulation $\tau_{\ell}(\Delta)$ for each $\Delta$ or the triangulation $\tau_{u}(\Delta)$ for each $\Delta$. Thus the hypersurface $Z$ is homeomorphic to a piecewise linear hypersurface $L$
obtained via the combinatorial patchworking construction using a triangulation which refines $\tau$ along a triangulation of a circuit as above. It remains to apply Theorem~\ref{patchwork}.  
\end{proof}

\section{Bounds for the number of non-defective faces}
\label{S:def}

In this section, we give bounds for the number of non-defective faces of a polytope $Q=\ch(\A)$ in terms of the dimension and codimension of $\A$.
Such bounds were already obtained in~\cite{FNR17} but as mentioned in the introduction the bounds in the cited paper are not correct. We first present families of polytopes with a number of non-defective faces which exceed drastically the bounds in~\cite{FNR17}. Then we present our new bounds.

\subsection{Configuration with many non-defective faces}

The polytopes described below are called Lawrence polytopes (see~\cite{BVSWZ} for a presentation).

Let \(m,k \in \mathbb{N}\). Set \(n=2m+k+1\). Let \( \A = \{v_1,\ldots,v_{m+k+1}\} \subset \mathbb{R}^m\) be \(m+k+1\) vectors such that any subset of \(m\) of them is affinely independent. In particular, they linearly generate \(\mathbb{R}^m\). As usual, we consider the associated  matrix \(\mA\) of size \((m+1)\times(m+k+1)\)
\[
\mA = \begin{pmatrix}
	1 & 1 & \cdots & 1 \\
	| & | & & | \\
	v_1 & v_2 & & v_{m+k+1}  \\
	| & | & & | 
\end{pmatrix}.
\]

We define the following set of points \(\mathcal{L} \subset \mathbb{R}^{n}\) given by its associated matrix \(\mA_{\mathcal{L}}\) defined by the blocks decomposition:
\begin{equation*}
	\mA_{\mathcal{L}} \defeq \begin{pmatrix}
		\mA & 0 \\
		I_{m+k+1} & I_{m+k+1}
	\end{pmatrix}
\end{equation*}
where \(I_{m+k+1}\) is the identity matrix of size \(m+k+1\) and \(0\) is the \(0\) matrix of size \((m+1)\times(m+k+1)\). Consequently, \(\mA_{\mathcal{L}}\) has size \((2m+k+2)\times(2m+2k+2)\), which is \((n+1)\times(n+k+1)\), has maximal rank, and corresponds to a homogeneous configuration \(\mathcal{L}_A \subset \mathbb{R}^{n}\)  (since the last \(m+k+1\) rows sum to the \(1\)'s row). 
The convex hull forms a Lawrence polytope. Let us denote its points (following the order of the columns) by \((\tilde{v}_1,\ldots,\tilde{v}_{m+k+1},\tilde{u}_1,\ldots,\tilde{u}_{m+k+1})\). 

Following Lemma 9.3.1 in~\cite{BVSWZ}, circuits of \(\mathcal{L}\) are exactly of the form \(\{\tilde{v}_{i} \mid i \in I\}\cup\{\tilde{u}_i \mid i \in I\}\) where \(\{v_i \mid i\in I\}\) is a circuit of \(\A\). Moreover they are also faces (and so non-defective faces) of the convex hull of \(\mathcal{L}_A\).

Consequently, this gives a counter-example to Proposition 3.4 in~\cite{FNR17}.
\begin{prop}\label{Lawrence}
	Given two integers \(k,n\) with \(n \geq k+1\), choosing \(m = \lfloor (n-k-1)/2\rfloor\), there exists a configuration of points of dimension \(n\) and codimension \(k\) with at least \(\binom{\lfloor(n+k+1)/2\rfloor}{k}\) non-defective faces.
\end{prop}

Let us show now that we can still obtain some non-trivial upper bounds for the number of non-defective faces.

\subsection{Upper bounds}

Let \(\A\) be a configuration of vectors of dimension \(n\) and codimension \(k\) given by its associated matrix \(\mA\). We assume furthermore that \(\A\) is not pyramidal. Let \(\mB\) be a matrix Gale dual to \(\mA\). We call \((B_i)_{i\in \A}\) the rows of \(\mB\). 

We want to bound from above the number of non-defective faces of \(\A\). In particular, it is sufficient to bound the number of subsets of \(\A\) which are not a pyramid.

For \(0 \leq \ell \leq k\), we define the sets
\[
E_\ell \defeq \{  \S \subset \A \mid \text{codim}(\S)=\ell \text{ and } \S \text{ is not a pyramid}\}.
\]
First notice that a simplex is a particular case of pyramid, hence  \(E_0\) is the empty set.

Let \(\mathcal{F}^{*}_r\) be the set of flats of \(\MatrAd\) of rank \(r\). Say differently, it corresponds to the set of subsets \(\S\) of \(\A\) such that \(\dim(\text{span}(\{B_i \mid i\in \S\})) = r\) and such that if \(B_j \in \text{span}(\{B_i \mid i\in \S\})\) then \(j\in I\) (closure property).

\begin{lem}
	For all \(0 < \ell \leq k\) 
	we have \(\S \in E_\ell\) if and only if \(\A\setminus\S \in \mathcal{F}^{*}_{k-\ell}\).
\end{lem}

\begin{proof}
	We show that this lemma directly follows from standard facts from matroid theory but, for readers unfamiliar with this theory, the lemma can also be easily proved by direct arguments of linear algebra.
	
	First, \(\rk^{*}(\A\setminus\S) = \lvert \A\setminus \S\rvert + \rk(\S) - \rk(\A) = (\lvert \A\rvert - \rk(\A)) - (\lvert S\rvert - \rk(\S))\). Hence any set \(\A\setminus \S\) has rank \(k-\codim(\S)\) in \(\MatrAd\).
	
	A subset \(\S\) is a pyramid if and only if there is an element \(e \in \S\) which does not belong to any circuit from \(\S\). That is to say, \(\S\) is a pyramid if and only if \((\MatrA \mid \S)^{*}\), the dual of the matroid restricted to \(\S\), contains a loop. By Theorem~2 (page 63 in~\cite{Welsh}), we know that \((\MatrA \mid \S)^{*} = \MatrAd \centerdot \S\) where \(\MatrAd \centerdot \S\) is the contraction of \(\MatrAd\) to \(\S\). Finally, we conclude (for example, by Exercise 3.2, page 64 in~\cite{Welsh}) that \(\S\) is a pyramid if and only if \(\A\setminus\S\) is not a flat of \(\MatrAd\).
\end{proof}

We know that \(E_\ell\) and \(\mathcal{F}^{*}_{k-\ell}\) have same cardinal. From any flat \(F\) of  \(\mathcal{F}^{*}_{k-\ell}\), we can extract an independent \(I \subset F\) (in \(\MatrAd\)) of rank \(k-\ell\) (obviously, we can retrieve the flat \(F\) from such an \(I\)). Such \(I\) has cardinal \(k-\ell\), consequently, the cardinal of  \(\mathcal{F}^{*}_{k-\ell}\) is bounded by the number of subsets of size \(k-\ell\) in \(\A\).

Moreover, we can slightly refine this upper bound when \(\S\) corresponds to a face of \(\A\). In this case, we saw in remark~\ref{rem:cofaceGale} that the origin belongs to the cone \(\sum_{a \notin \S} \R_{>0} B_a\). In particular, the rows \((B_a)_{a \notin \S}\) are linearly dependent, and for each row \(B_a\), there is a base of \((B_a)_{a \notin \S}\) which does not contain \(B_a\) (i.e., the restricted matroid has no coloop). It implies there exist at least \(k-\ell+1\) bases of \((\mB_a)_{a \notin \S}\). 

\begin{prop}\label{prop:number_nd_faces}
	The number of non-defective faces of \(\A\) of codimension \(\ell\) is bounded from above by
	\[
	\frac{1}{k-\ell+1}\binom{n+k+1}{k-\ell} = \frac{1}{n+k+2}\binom{n+k+2}{k-\ell+1}.
	\]
	
	Consequently, the total number of non-defective faces is bounded by 
	\[
	\sum_{j=0}^{k-1} \frac{1}{j+1} \binom{n+k+1}{j} \leq \binom{n+2k}{k-1}.
	\]
\end{prop}

\section{Bounds for the number of connected components of hypersurfaces}
\label{Final bounds}
We first present our bounds in the general case. Then we refine the previous bounds when the codimension equals two and show that the resulting bound is sharp
in dimension two.

\subsection{General case}
Our general bounds are as follows.

\begin{theo} \label{thm_generic_bound}
	Let $\A$ be a finite set in $\R^{n}$ such that $\codim \A= k$. Let \(m\) be the dimension of the basis of \(\A\).
	For any smooth real polynomial $f=\sum_{a\in \A}c_a x^{a}$, 
	we have
	\begin{equation}
		b_{0}(V_{>0}(f)) \leq  1+k+\frac{e^2+3}{8(m+k+2)} \sum_{j=0}^{k-1} \binom{m+k+2}{k-j}  2^{\binom{j}{2}} (m+1)^{j}.
	\end{equation}
	The bound can be replaced by the weaker but simpler expression
	\begin{equation}
		b_{0}(V_{>0}(f)) \leq 8 (m+1)^{k-1} 2^{\binom{k-1}{2}}.
	\end{equation}
\end{theo}
\begin{proof}
	Consider a path $(f_{t})_{t \in ]0,+\infty[}=(c_{a}t^{h_{a}})_{t \in ]0,+\infty[}$. 
	We might choose $h$ compatible with \(\A\) and generic enough so that any face \(F\) of \(\Ah = \{(a,h_a) \mid a \in\A\}\) is a simplex with set of vertices \(\Ah \cap F\). This gives two polyhedral triangulations of $Q$ obtained by projecting the lower part and the upper part of the convex hull of $\Ah$. 
	We might furthermore perturb slightly the coefficients $c_{a}$ so that the path \(f_t\) intersects \(\nabla\) only at smooth points (see Lemma~\ref{lem:intersection_nablas}) and that the condition in Proposition \ref{Jens} is always satisfied (see Remark \ref{R:no restrictive}). 
	For any face \(F\) of \(Q\), let us consider the set $T_F$ of $t \in \; ]0,+\infty[$ for which there exists $x \in \R_{>0}^{n}$ such that $(x,t)$ is a solution of~\eqref{S_F}. By Corollary~\ref{cor:pyramidal}, \(T_F\) is empty as soon as \(\AF\) is pyramidal. 
	By Remark~\ref{rem:npfaces}, it is sufficient to consider the faces of \(\A'\) where \(\A'\) is the basis of \(\A\).
	Let \(T = \bigcup_{F\text{ face of }Q'} T_F\) (where $Q'$ is the convex hull of $\A'$).
	Then $T$ is finite and $b_{0}(V_{>0}(f_{t}))$ may only change when $t$ passes through a value $t \in T$, in which case  $b_{0}(V_{>0}(f_{t}))$ increases or decreases by at most $1$ (Propositions \ref{isotopic} and \ref{Jens}). 
	Thus setting $n_{0}=b_{0}(V_{>0}(f_{t}))$ for $0 <t < \min(T)$ and $n_{+\infty}=b_{0}(V_{>0}(f_{t}))$ for $t > \max(T)$, we get	
	\begin{equation}\label{T+patchwork}
		b_{0}(V_{>0}(f) )\leq \frac{|T|+n_{0}+n_{+\infty}}{2}.
	\end{equation}
	From Theorem~\ref{patchwork}, we have $\max(n_{0}, n_{+\infty}) \leq k+1$, hence 
	$b_{0}(V_{>0}(f)) \leq \frac{|T|}{2}+k+1$. 	
	
	For any face \(F\) of dimension \(\nu\) such that \(\AF\) has codimension \(\kappa\), by Corollary~\ref{boundSF}, we know that \(\lvert T_F \rvert \leq \frac{e^2+3}{4} 2^{\binom{\kappa-1}{2}} (\nu+1)^{\kappa-1} \). 
	Then
	\begin{align*}
		\lvert T\rvert & = \sum_{\nu=1}^{m} \sum_{\kappa=1}^{k} \sum_{\stackrel{F\text{ non-def. face of }Q}{\dim(F)=\nu, \codim(F)=\kappa}} \lvert T_F \rvert \\
		& \leq \sum_{\nu=1}^{m} \sum_{\kappa=1}^{k} \sum_{\stackrel{F\text{ non-def. face of }Q}{\dim(F)=\nu, \codim(F)=\kappa}} \frac{e^2+3}{4} 2^{\binom{\kappa-1}{2}} (\nu+1)^{\kappa-1} \\
		& \leq \frac{e^2+3}{4} \sum_{\kappa=1}^{k} \frac{1}{m+k+2}\binom{m+k+2}{k-\kappa+1}  2^{\binom{\kappa-1}{2}} (m+1)^{\kappa-1} \\
		& = \frac{e^2+3}{4} (m+1)^{k-1} \\
		& \quad\quad\quad\quad \sum_{\kappa=1}^{k} \frac{(m+k+1)(m+k)\cdots(m+\kappa+2)}{(k-\kappa+1)! (m+1)^{k-\kappa}}  \left(2^{\binom{k-1}{2} - \binom{k-2}{1} - \cdots - \binom{\kappa-1}{1}}\right) \\
		& = \frac{e^2+3}{4} (m+1)^{k-1} 2^{\binom{k-1}{2}} \sum_{\kappa=1}^{k} \frac{1}{(k-\kappa+1)!} \prod_{j=0}^{k-\kappa-1} \frac{m+\kappa+j+2}{2^{\kappa+j-1} (m+1)} \\
		& \leq \frac{e^2+3}{4} (m+1)^{k-1} 2^{\binom{k-1}{2}} \frac{5}{2} (e-1).
	\end{align*}
	We used the fact that the product \(\prod_{j=0}^{k-\kappa-1} \frac{m+\kappa+j+2}{2^{\kappa+j-1} (m+1)}\) is at most \(5/2\) (reached for \(\kappa=1\), \(k=3\), and \(m=1\)).
	Finally, notice that since \(m,k\ge 1\), \(\frac{e^2+3}{4} (m+1)^{k-1} 2^{\binom{k-1}{2}} \frac{5}{2} e + 2k+2 \le  16 (m+1)^{k-1} 2^{\binom{k-1}{2}}\).
	\end{proof}

\subsection{The codimension two case}

In the codimension two case, if $\rk \mAh =\rk \mA+1$ then the support of the critical system of the associated Viro polynomial has codimension \(1\) and so is a circuit (or a pyramid over a circuit). As sharp bounds are known in this case, we can refine our result.

Let $\A$ be a finite set in $\R^{n}$ such that $\dim \A=n \geq 2$ and $\codim \A=2$. 
Consider any real polynomial $f=\sum_{a\in \A}c_a x^{a}$.
To obtain more precise bounds we will use here non generic height functions.  For $\alpha \in \A$ we will consider the height vector $h^{\alpha} = (h_a)_{a\in \A}$ defined by $h_a=1$ if $a=\alpha$ and $h_a=0$ otherwise. Notice that such function \(h^{\alpha}\) is not compatible (see Definition \ref{de:h}). Indeed, if \(F\) is a face which does not contain \(\alpha\), then \(\rk{\mA_F^h} = \rk{\mAF}\). However, if \(F\) is a non-pyramidal face which contains \(\alpha\), the row vector \((h_a)_{a \in \AF}\) does not lie in the row span of \(\mAF\).
Otherwise, \(\AF\) would be pyramidal over \(\AF\setminus\{\alpha\}\). 
As in the proof of Theorem~\ref{thm_generic_bound}, for any face \(F\) of \(Q\) we denote by $T_F$ the set of $t \in \; ]0,+\infty[$ for which there exists $x \in \R_{>0}^{n}$ such that $(x,t)$ is a solution of~\eqref{S_F} and set \(T = \bigcup_{F\text{ face of } Q'} T_F\).
Set $n_{0}=b_{0}(V_{>0}(f_{t}))$ for $0 <t < \mbox{min} \,  T$ and $n_{+\infty}=b_{0}(V_{>0}(f_{t}))$ for $t > \mbox{max} \, T$.

\begin{lem}\label{bound nonpatchwork}
We have $n_{0} \leq 2$ and $n_{+\infty} \leq 2$.
\end{lem}
\begin{proof}
This follows from Proposition~\ref{non patchwork}.
\end{proof}
We are now able to state the main result of this section.

\begin{theo} \label{Fewnomial boundbis}
Let $\A$ be a finite set in $\R^{n}$ of dimension \(d \geq 2\) such that $\codim \A=2$ and its basis \(\A'\) has dimension \(m\). 
Let $u \in \A'/{\sim}$ be any equivalence class (see Subsection \ref{S:smallcodim} for the definition of $\A'/{\sim}$). Then,
for any real polynomial $f=\sum_{a\in \A} c_a x^{a}$, we have
\begin{equation}\label{mainbis}
b_{0}(V_{>0}(f)) \leq \left\lfloor \frac{m-|u|}{2}\right\rfloor+3.
\end{equation}
\end{theo}

\begin{proof}
Choose $\alpha \in u$ and consider the height vector $h$ defined by $h_a=1$ if $a=\alpha$ and $h_a=0$ otherwise. 

Let \(Q' = \ch(\A')\).
Consider the path $(f_{t})_{t \in ]0,+\infty[}=(c_{a}t^{h_{a}})_{t \in ]0,+\infty[}$.
By our choice of $h$ we see that if $F$ is a face of $Q'$ which does not contain $\alpha$, then $f^{F}_{t}$ does not depend on $t$ i.e.\ is equal to $f^{F}$. Perturbing slightly the coefficients $c_a$ if necessary, we may assume that $(f_{t})_{t \in ]0,+\infty[}$ intersects \(\nabla\) at smooth points (in order to use Proposition~\ref{Jens}) and that for all faces $F$ of $Q'$ which does not contain $\alpha$ the hypersurface $f^{F}=0$ has no singular point in the positive orthant. If follows that only critical systems~\eqref{S_F} associated to non-defective faces $F$ of $Q'$ containing $\alpha$ can have positive solutions. For such faces, it holds that \(\rk{\mA}_{F}^{h} = \rk{\mAF}+1\).

Recall that for any face \(F\) of \(Q\) we denote by $T_F$ the set of $t \in \; ]0,+\infty[$ for which there exists $x \in \R_{>0}^{n}$ such that $(x,t)$ is a solution of~\eqref{S_F} and that \(T = \bigcup_{F\text{ face of }Q'} T_F\).
Thus $T\setminus T_{Q'}$ consists of the elements $t$ of $T$ for which there exists a circuit $\AF$ being a proper face of $\A'$ (in particular \(\dim \AF < m\)) and such that there exists $x \in \R_{>0}^{n}$ with $(x,t)$ a solution of~\eqref{S_F}.
Then, $|T\setminus T_{Q'}|$ does not exceed the total number $N_{\alpha}$ of circuits $\A_{F}\subset \A'$ such that $\dim \A_{F}\leq m-1$ and $\alpha \in \A_{F}$. By Lemma~\ref{colinear}, the number $N_{\alpha}$ does not exceed the number $d$ of equivalences classes $v \in \A'/{\sim}$ such that $|v| \geq 2$ and $v \neq u$.
Letting $s+1=\lvert \A'/{\sim} \rvert$, we get $|\A'|=m+3 \geq \lvert u\rvert+N_{\alpha}+s$. Thus,
\begin{equation}\label{easy}
N_{\alpha}+s \leq m+3-|u|.
\end{equation}
We have $T_{Q'} \leq \signvar(b_{u_0}, b_{u_1},\ldots,b_{u_{s}})$ by Proposition~\ref{systeme principal}. Thus from the inequality~\eqref{T+patchwork} and Lemma~\ref {bound nonpatchwork} we get
\begin{equation}\label{refined}
b_{0}(V_{>0}(f)) \leq \frac{N_{\alpha}+\signvar(b_{u_0}, b_{u_1},\ldots,b_{u_{s}})}{2}+2.
\end{equation}
It turns out that the term $b_{u}$ in the previous sequence vanishes. Indeed,  recall that $b$ is a vector in $\mbox{Ker} \, \mAh \subset \mbox{Ker} \, A$. Thus $b_{\alpha}=0$ and there exists $\delta$ (a column vector) such that $b=\mB\cdot\, \delta $. It follows that $\mB_{\alpha} \cdot \, \delta=0$. Now, if $\ell \in u$, then $\mB_{\ell}$ and $\mB_{\alpha}$ are colinear, and thus $b_{\ell}=\mB_{\ell}\cdot  \delta=0$. It follows that $b_{u}=\sum_{\ell \in u}b_{\ell}=0$. Therefore, the sequence $(b_{u_0}, b_{u_1},\ldots,b_{u_{s}})$ has at most $s$ non-zero terms. Thus, 
$b_{0}(V_{>0}(f)) \leq \frac{N_{\alpha}+s-1}{2}+2$ and using~\eqref{easy} we get
\[b_{0}(V_{>0}(f)) 
\leq \frac{m-|u|}{2}+3.\qedhere
\]
\end{proof}

Theorem~\ref{Fewnomial boundbis} has the following very simple reformulation using Lemma~\ref{colinear}.

\begin{theo}\label{boundusingcircuits}
If \(\A\) is a finite set of \(\R^n\) of codimension \(2\), then $b_{0}(V_{>0}(f)) \leq \left\lfloor \frac{r-1}{2} \right\rfloor+3$ where 
$r$ is the minimal dimension of a circuit ${\mathcal C} \subset \A$.
\end{theo}
\begin{proof}
If \(\dim \A =1\), then \(\lvert \A \rvert =4\), and \(3\) is a bound on the number of positive solutions (which are the connected components). Assume \(\dim \A \geq 2\). Since ${\mathcal C} $ is not a pyramid, we have ${\mathcal C} \subset \A'$. Consider $u=\A' \setminus {\mathcal C} \subset \A'$. Then $u$ is an equivalence class of
$\A^\prime/{\sim}$ by Lemma \ref{colinear}.
Then Theorem \ref{Fewnomial boundbis} yields $b_{0}(V_{>0}(f)) \leq \frac{m-|u|}{2}+3$ and it remains to use that $|u|=
|\A'| -|{\mathcal C}|=m+3-(r+2)$.
\end{proof}

We next show that the bound in Theorem \ref{Fewnomial boundbis} (equivalently in Theorem \ref{boundusingcircuits}) is sharp for $n=2$.

\begin{theo}\label{sharp n=2}
Let $\A \subset \R^{2}$ be a set of at most five points in $\R^2$ which do not belong to a line. Then, for any real polynomial $f=\sum_{a\in \A}c_a x^{a}$, we have $b_{0}(V_{>0}(f)) \leq 3$
Moreover, we have $b_{0}(V_{>0}(f))= 3$ for the polynomial $f(x,y) = 1+ x^{4}-xy^{2}-x^{3}y^{2}+0.76x^{2}y^{3}$ (see Figure~\ref{fig:threeCompo}).
\end{theo}

\begin{figure}[ht]
	\centering
	\includegraphics[scale=0.2]{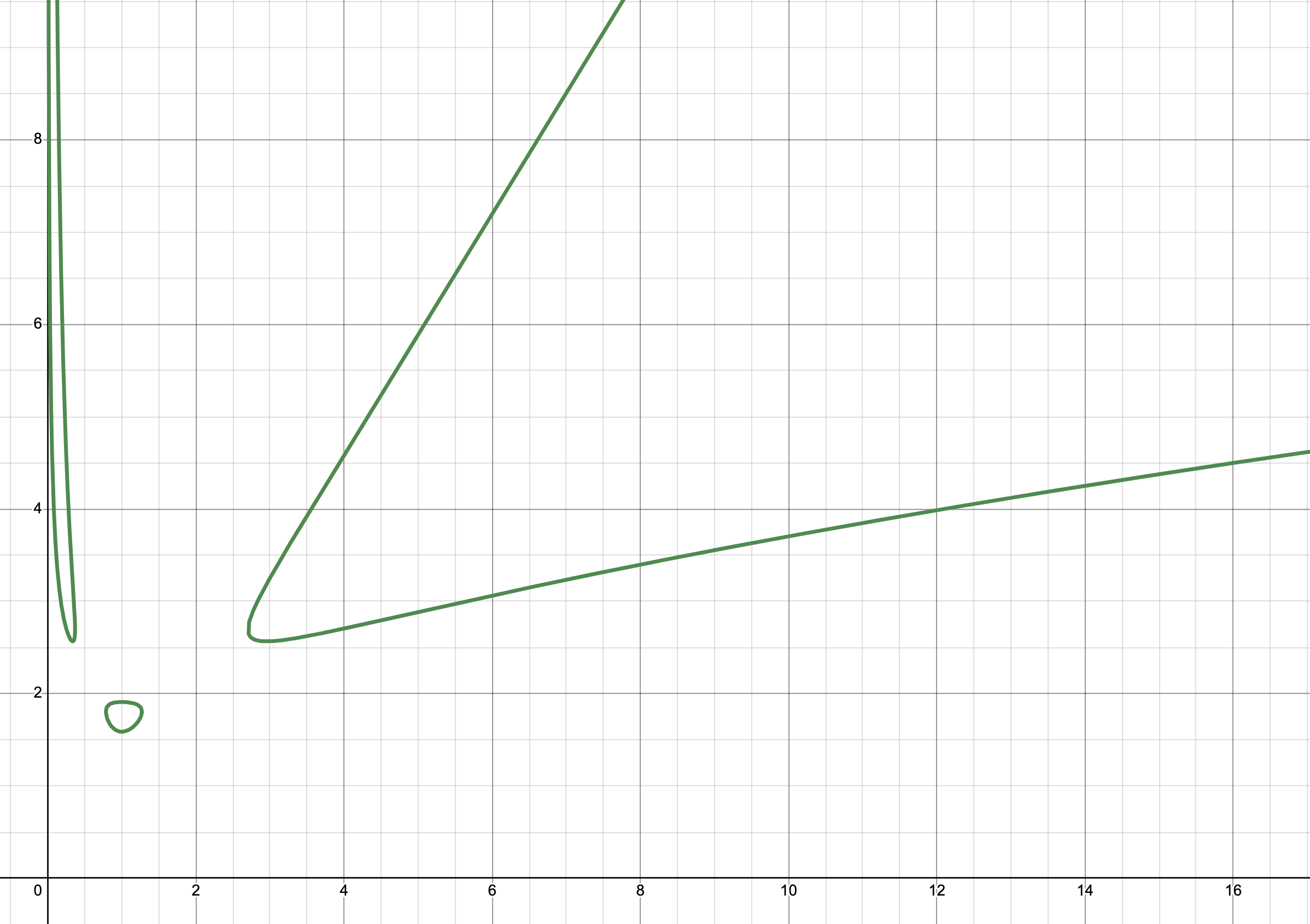}
	\label{fig:threeCompo}
	\caption{The positive zero set of \(f = 1+ x^{4}-xy^{2}-x^{3}y^{2}+0.76x^{2}y^{3}\) has three connected components.}
\end{figure}

\begin{rem}
The assumption that the points do not belong to a line cannot be dropped: the curve defined by the polynomial $f(x_{1},x_{2})=
(x_1-1)(x_1-2)(x_1-3)(x_1-4)$ has $4$ connected components in the positive orthant.
\end{rem}

\begin{rem}
By Theorem \ref{patchwork}, the curve defined by the polynomial $f(x,y) = 1+ x^{4}-xy^{2}-x^{3}y^{2}+0.76x^{2}y^{3}$ in the positive orthant cannot be obtained  by the combinatorial patchworking construction.
\end{rem}

No datasets were generated or analysed during the current study.
\bibliographystyle{plain}
\bibliography{DescartesMultivariate} 

\appendix

\section{Monomial change of coordinates}
 \label{monomial}

For reader's convenience we recall well-known results about monomial maps, more specifically, monomial change of coordinates.
 
 A {\em monomial change of coordinates} is a map $\varphi_{L} : x= (x_{1},\ldots,x_{n}) \mapsto y=(y_{1},\ldots,y_{n)}$ where $y_i=\prod_{j=1}^{n}x_{j}^{L_{ij}}$ for $i=1,\ldots,n$ and $L=(L_{ij}) \in \textrm{GL}_n(\R)$. This provides a diffeomorphism $\varphi_{L} : (\R_{>0})^{n} \rightarrow (\R_{>0})^{n}$. Moreover $\varphi_{L}^{-1} = \varphi_{L^{-1}}$. An easy way to see that is to conjugate $\varphi_{L}$ by the logarithm map $\Log: (\R_{>0})^{n} \rightarrow \R^{n}$ (which is a diffeomorphism) and observe that $\Log(y)=L \cdot \Log(x)$ (here and after $y=\varphi_{L}(x)$ and the vectors are seen as column vectors).
 It follows that for any $a \in \R^n$ we have $x^a= \prod_{i=1}^n \prod_{j=1}^n y_j^{L_{ij}^{-1} a_i} = y^{\ell(a)}$ where $\ell$ is the linear map associated to the transpose of $L^{-1}$, that is,  $\ell(a)=(L^{T})^{-1} \cdot a$.

 Then, if $f(x)=\sum_{a \in \A} c_{a} x^{a} \in \R[x]$, we get that $f=g \circ \varphi_L$, where $g$ is the polynomial $g(y)=\sum_{a \in \A} c_a y^{\ell(a)}$.
  
 \begin{lem}
  \label{change}
 Consider a polynomial $f(x)=\sum_{a \in \A}c_{a} x^{a}$ in $n$ variables $x=(x_{1},\ldots,x_{n})$.
Let $\varphi_L: (\R_{>0})^{n} \rightarrow (\R_{>0})^{n}$ be the monomial map
given by  $y_i=\prod_{j=1}^{n}x_{j}^{L_{ij}}$ for $i=1,\ldots,n$, and $L=(L_{ij}) \in \textrm{GL}_n(\R)$. Set $g(y)=\sum_{a \in \A}c_{a} y^{\ell(a)+\tau}$ where $\ell:\R^{n} \rightarrow \R^n$
is the linear map associated to $(L^{T})^{-1}$ and \(\tau \in \R^n\). 
Then the map \(x \mapsto \varphi_L(x)\) is a diffeomorphism which sends the set \(V_{>0}(f,x_1\frac{\partial f}{\partial x_1}, \ldots, x_n\frac{\partial f}{\partial x_n})\) onto $V_{>0}(g, y_1\frac{\partial g}{\partial y_1}, \ldots, y_n\frac{\partial g}{\partial y_n})$. 
\end{lem}
\begin{proof}
We can assume that \(\tau=0\) (see Remark~\ref{translation}).
By construction of \(\varphi_L\) and \(\ell\), we have  $f(x)=0$ if and only if $g(\varphi_L(x))=0$. It remains to see that the first partial derivatives with respect to $y$ correspond via logarithmic change of coordinates (at the source and at the target, see the discussion above) to directional derivatives in the variables $x$ along a basis of $\R^{n}$ determined by $L$.
\end{proof} 

 \begin{rem}\label{affine}
	Any affine transformation $\psi: \R^{n} \rightarrow \R^{n}$ corresponds to an invertible matrix $T_{\psi} \in \R^{(n+1) \times (n+1)}$ with first row $(1,0,\ldots,0)$ via the rule
	$$
	T_{\psi} \cdot \left(\begin{array}{c}
		1 \\
		a
	\end{array}\right)=
	\left(\begin{array}{c}
		1 \\
		\psi(a)
	\end{array}\right)
	.$$
	The first column of $T_{\psi}$ gives the translation vector of $\psi$ while the lower right matrix in $\R^{n \times n}$ is the matrix associated to its linear part.
\end{rem}

\section{Useful lemma and notations}\label{sec:App_prelim}
We first give some notations. Let $\textrm{M}$ be any square matrix. We will denote by \(\lvert \textrm{M}\rvert \) or by $\det(M)$ its determinant.
Moreover \(\textrm{M}_{I}^J\) is the submatrix from \(\mathrm{M}\) we get by keeping the rows from \(I\) and the columns from \(J\). We similarly use \(\mathrm{M}_I\) and \(\mathrm{M}^{J}\). Also \(\mathrm{M}_{\setminus i}\) and \(\mathrm{M}^{\setminus j}\) means that we remove respectively the \(i\)th row and the \(j\)th column from \(\mathrm{M}\). We will use the notation $\Diag_{u \in [1,b]}(v_u)$ for the diagonal matrix with diagonal coefficients $v_1,\ldots,v_b$.
The transpose of a matrix $A$ will be often denoted by $A^{T}$

    We start with an easy lemma which will be used several times in the proofs below.

\begin{lem}\label{lem_sum_sigma}
	Let \(\mA\) and \(\mC\) be two matrices of \(\R^{a\times b}\) with \(a \leq b\). Let \(v\) be a vector in \(\ker(\mC) \subseteq \R^b\). Assume that the row of \(1\)'s is in the row-span of \(\mA\). Then
	\begin{align*}
		\sum_{\substack{\sigma \subseteq [1,b] \\ \lvert \sigma \rvert =a}} \det(\mA^{\sigma}) \det(\mC^{\sigma}) \prod_{u \in \sigma} v_u = 0.
	\end{align*}
\end{lem}

\begin{proof}
	By Cauchy-Binet Formula, the sum equals 
	\begin{align*}
		\det(\mC \cdot \Diag_{u \in [1,b]}(v_u) \cdot \mA^{T} ).
	\end{align*}
Since there exists a non-zero column vector \(\rho \in \R^a\) such that \(\mA^{T} \cdot \rho\) is the column of \(1\)'s, it implies that \(\Diag_{u \in [1,b]}(v_u) \cdot \mA^{T} \cdot \rho =v\). In particular \(\rho\) lies in the kernel of \(\mC \cdot \Diag_{u \in [1,b]}(v_u) \cdot \mA^{T}\). 
\end{proof}

\section{Proof of Lemma~\ref{lem:intersection_nablas}}\label{sec:proof_intersection_nablas}

Let \(F_1\) and \(F_2\) be two distinct faces of Q. 
We keep notations introduced in Section~\ref{sec:App_prelim}. It is sufficient to prove that
${\overset{\frown}{\nabla}}_{\A_{F_1}} \cap {\overset{\frown}{\nabla}}_{\A_{F_2}}$ has codimension at least two in \(\R^{\A}\), where
${\overset{\frown}{\nabla}}_{\A_{F_i}}=\{c \in (\R^*)^{\A} \mid (c_{a})_{a \in \A_{F_i}} \in \accentset{\circ} \nabla_{\A_{F_i}}\}$.
Denote by \(d_1\) and \(d_2\) the dimensions of $F_1$ and $F_2$, respectively. Moreover, set
\(s_1 = \lvert \A_{F_1}\rvert\) and \(s_2 = \lvert \A_{F_2}\rvert\).
Assume without loss of generality that $F_i$ is contained in a coordinate subspace of dimension $d_i$ for $i=1,2$
(applying an affine transformation to $\A$ if necessary).
Then the matrix \(\mA_{F_i}\) has $n-d_i$ zero rows. Denote by $\mA_i$ the matrix of size \((d_i+1)\times s_i\) and rank \(d_i+1\) obtained by removing the zero rows of \(\mA_{F_i}\).
From Horn-Kapranov uniformization, we know that \(\accentset{\circ}{\nabla}_{\A_{F_1}}\) is parametrized by \(\Phi_{\A_{F_1}}\) over \((\mathbb{R}_{>0})^{d_1} \times \mathbb{R}^{s_1-d_1-1}\).
Notice that if the set \(\A_{F_1}\) is defective, then by definition \(\accentset{\circ}{\nabla}_{\A_{F_1}}\), and so \(\overset{\frown}{\nabla}_{\A_{F_1}}\), have codimension at least two and the result is clear. So we assume in the following that \(\A_{F_1}\) is not defective. By~\cite{For19}, this implies that the cuspidal form \(P_{\A_{F_1}}\) is not identically zero. In particular, the image by \(\Phi_{\A_{F_1}}\) of \((\mathbb{R}_{>0})^{d_1} \times \{z \in \mathbb{R}^{s_1-d_1-1} \mid P_{\A_{F_1}}(z) = 0\}\) has codimension at least two in \(\R^{\A_{F_1}}\).
Thus it is enough to prove that the set 
$$
\mathcal{S} = (\R^*)^{\A} \cap
\overset{\frown}{\nabla}_{\A_{F_1}} \cap
\overset{\frown}{\nabla}_{\A_{F_2}}\setminus \{ c \mid \exists x,z, \ \Phi_{\A_{F_1}}(x,z)=c_{\mid \A_{F_1}} \mbox{ and } P_{\A_{F_1}}(z) = 0 \}
$$ has codimension at least two in \(\R^{\A}\) (where we use the notation \(c_{\mid \A_{F_1}}=(c_a)_{a \in \A_{F_1}}\)).
Consider the block-diagonal matrix of size \((d_1+d_2+2)\times(s_1+s_2)\)
\[
	\tilde{\mA}= \begin{pmatrix} \mA_1 & 0 \\ 0 & \mA_2\end{pmatrix}.
\]
Since \(\rk(\mA_i)=d_i+1\) we have \(\rk(\tilde{\mA}) = d_1+d_2+2\). 
Let \(\mathcal{W}_2 \subseteq \A_{F_2}\) of size \(d_2+1\) such that the maximal minor of \(\mA_2\) corresponding to \(\mathcal{W}_2\) is invertible.

Consider the map $\psi:(\R^*)^{ \A}\times(\R_{>0})^{d_1}\times(\R_{>0})^{d_2} \rightarrow \R^{d_1+d_2+2}$ defined by
$$
\psi(c,x,y)=\tilde{\mA} \cdot \left(\begin{array}{c} u_1 \\ u_2 \end{array} \right)=\left(\begin{array}{c} \mA_1 u_1 \\ \mA_2 u_2 \end{array} \right),$$ where $u_1= (c_a x^a)_{a \in \A_{F_1}}$ and  $u_2= (c_a y^a)_{a \in \A_{F_2}}$ are written as column vectors.
We have \(\psi(c,x,y)=0\) if and only if the hypersurface defined by the polynomial with coefficients $(c_a)_{a \in \A_{F_1}}$
has a singularity at $x$ and the hypersurface defined by the polynomial with coefficients $(c_a)_{a \in \A_{F_2}}$
has a singularity at $y$. Therefore, the image of 
\(\psi^{-1}(0)\) by the projection onto the factor \((\R^*)^{\A}\) equals \((\R^*)^{\A} \cap
{\overset{\frown}{\nabla}}_{\A_{F_1}} \cap
{\overset{\frown}{\nabla}}_{\A_{F_2}}\).
Hence, the result follows from the following lemma.

\begin{lem}
If \((c,x,y) \in \psi^{-1}(0)\) and $c \in \mathcal{S}$ then the dimension of \(\psi^{-1}(0)\) at \((c,x,y)\) is at most \(\lvert \A\rvert -2\).
\end{lem}

\begin{proof}
Let \((c,x,y)\in \psi^{-1}(0)\) be such that $c \in \mathcal{S}$. Since \(\psi(c,x,y)=0\), we have \(\mA_{1} \cdot \left(c_ax^a\right)_{a \in \A_{F_1}} =0\), thus
there exists \(z \in \R^{s_1-d_1-1}\) such that \(\Phi_{\A_{F_1}}(x,z)=c_{\mid \A_{F_1}}\) (and \(\mB_1 \cdot z = \left(c_a x^a\right)_{a \in \A_{F_1}}\),
where \(\mB_1\) is some Gale dual matrix of \(\mA_1\)). By assumption, \(P_{\A_{F_1}}(z) \neq 0\).

We show that the Jacobian matrix of \(\psi\) at \((c,x,y)\) has full rank by exhibiting a full rank \((d_1+d_2+2)\)-minor. 
By symmetry, we might assume that \(F_1\) contains a vertex \(\overline{a}\) which is not in \(F_2\).
Consider the minor obtained by selecting the columns associated to the following derivatives: the one with respect to \(c_{\overline{a}}\), the \(d_1\) toric derivatives with respect to \(x\), and the \(d_2+1\) derivatives with respect to \((c_a)_{a \in \mathcal{W}_2}\). Denote by \(\mathcal{D}\) the determinant of this minor. 
We get
\begin{align*}
	\mathcal{D} & = \det\left( \tilde{\mA} \cdot  
	\begin{bmatrix} 
		\begin{matrix} 
			x^{\overline{a}} \\ 0 \\ \vdots \\ 0
		\end{matrix} 
		&
		\begin{matrix}
			\Diag_{a \in \A_{F_1}}(c_a x^a ) \cdot 
			(\hat{\mA}_1)^{T}
		\end{matrix}  
		& 
		\begin{matrix}
			(\Diag_{a \in \A_{F_1}}(x^a))^{\mathcal{W}_2} 
		\end{matrix}
		\\
		\hline \\
		\bigzero & \bigzero &
		\begin{matrix}
			(\Diag_{a \in \A_{F_2}}(x^a))^{\mathcal{W}_2} 
		\end{matrix} 
	\end{bmatrix}\right) \\	
\end{align*}
where $\hat{\mA}_1$ is the matrix $\mA_1$ without its first row of one's.
Then, by writing $\mathcal{D} = \det\left(\tilde{\mA} \cdot \mR\right)$ and applying Cauchy Binet Formula we obtain
$\mathcal{D}=\sum_{|I|=d_1+d_2+2} \det\left( \tilde{\mA}^I \right) \cdot \det\left(\mR_I \right)$.
The block diagonal structure of \(\tilde{\mA}\) implies that the sum can be restricted to sets $I$ containing \(d_1+1\) elements of \(\A_{F_1}\) and \(d_2+1\) elements of \(\A_{F_2}\), while the structure of the matrix $\mR$ implies that the sum can be restricted to sets $I$ containing \(\overline{a} \in \A_{F_1}\) and \(\mathcal{W}_2 \subseteq \A_{F_2}\). Thus we get

\begin{align*}
	\mathcal{D} & =  \sum_{\substack{\sigma \subseteq \mathcal{A}_{F_1}\setminus\{\overline{a}\} \\ \lvert \sigma \rvert = d_1}}
	(\det{\mA_1^{\{\overline{a}\} \mathbin\Vert \sigma }}) (\det{\mA_2^{\mathcal{W}_2}})
	(\det(\hat{\mA}_1)^{\sigma}) x^{\overline{a}}
	\prod_{a \in \sigma} c_a x^a \prod_{a \in \mathcal{W}_2} y^a	\\
	&= 
	\det(\mA_2^{\mathcal{W}_2}) x^{\overline{a}} \prod_{a \in \mathcal{W}_2} y^a
	\sum_{\substack{\sigma \subseteq \mathcal{A}_{F_1} \\ \lvert \sigma \rvert = d_1}}
	(\det{\mA_1^{\{\overline{a}\} \mathbin\Vert \sigma }}) 
	(\det(\hat{\mA}_1)^{\sigma}) 
	\prod_{a \in \sigma} c_a x^a,
\end{align*}
where \(u \mathbin\Vert v\) stands for the concatenation of \(u\) and \(v\).
Expanding the second determinant along the \(\{\overline{a}\}\)-column, the above sum equals
\begin{align*}
	&\sum_{\ell = 0}^{d_1} (-1)^{\ell} (\mA_1)_{\ell+1,\overline{a}}
	\sum_{\substack{\sigma \subseteq \mathcal{A}_{F_1} \\ \lvert \sigma \rvert = d_1}} 
	\det{((\mA_1)_{\setminus \{\ell+1\}}^{\sigma }}) \cdot 
	\det((\hat{\mA}_1)^{\sigma}) \cdot
	\prod_{a \in \sigma} c_a x^a.
\end{align*}
Since the vector \((c_a x^a)_{a \in \mathcal{A}_{F_1}}\) is in the kernel of \(\hat{\mA}_1\), the inside sum is zero when \(\ell \neq 0 \) by Lemma~\ref{lem_sum_sigma} and equals
\[
    \sum_{\substack{\sigma \subseteq \mathcal{A}_{F_1} \\ \lvert \sigma \rvert = d_1}} 
	\left(\det(\hat{\mA}_1)^{\sigma}\right)^2
	\prod_{a \in \sigma} c_a x^a = P_{\A_{F_1}}(z)
\]
otherwise (using that \(\mB_1 \cdot z = \left(c_a x^a\right)_{a \in \mathcal{A}_{F_1}}\)).
Consequently,
\[
	\mathcal{D} = (\det{\mA_2}^{\mathcal{W}_2}) \left(x^{\overline{a}} \prod_{a \in \mathcal{W}_2} y^a\right) P_{\mathcal{A}_{F_1}}(z) \neq 0. \hfill \qedhere
\]
\end{proof}

\section{Proof of Claim~\ref{clm:PA}}\label{app:PA}\label{sec:proof_det_jacobian}

We prove the following claim which is used for proving Proposition~\ref{simple}. We keep notations introduced in Section~\ref{sec:App_prelim}.

\begin{clm*}[Restating of Claim~\ref{clm:PA}]
	For all \((y_1,\ldots,y_k) \in {{\mathscr{C}}_D^{\nu}}\), we have
	\[
	\det(J) = \gamma \left(\prod_{a \in \A} \langle B_a,y \rangle^{-1+\sum_{p=1}^{k-1}\lambda_{a,p}}\right) \cdot P_{\mA}(y) \cdot \left( \sum_{j=1}^k y_i^2\right) \cdot \left( \sum_{a \in \A} h_a \langle B_a, y\rangle \right)
	\]
	where \(\gamma\) is a non-zero real constant which does not depend on \(y\).
\end{clm*}

\begin{proof}
	We will apply several times Laplace expansion, so for \(I \subset [1,p]\) let us denote by \(\varepsilon(I)\) the value \(\sum_{i \in I} i +\sum_{i=1}^{\lvert I \rvert} i\).
	To be coherent with this previous convention, we use here an ordering of \(\A = \{a_1,\ldots,a_{n+k+1}\}\)  (starting the numbering by \(1\) and not \(0\) as in the remainder of the paper). Thus for \(I \subset \A\), we write \(\varepsilon(I)\) for \(\varepsilon(\{j \in [1,n+k+1] \mid a_j \in I\})\).
	
	Finally, if \(V \in M_{p+q,q}\) is a matrix Gale dual to a matrix \(U \in M_{p,p+q}\), then we define \(\gamma_U\) as the constant verifying that for all \(I \subset [1,p+q]\) with \(\lvert I \rvert = q\), \(\lvert V_I \rvert = (-1)^{\varepsilon(^cI)} \cdot \gamma_V \cdot \lvert U^{^cI} \rvert\).

	We expand the expression of \(\det({J})\) to get the right hand side of the claimed identity.	
	Since \(B_a = c_a D_a\), 
	\[
	\frac{\partial \phi_i}{\partial y_j} = \sum_{a \in \A} \frac{\lambda_{a,i} \mD_{a,j}}{\langle D_{a},y\rangle}
	\prod_{\ell \in \A} \langle D_\ell , y \rangle^{\lambda_{\ell,i}} = \sum_{a \in \A} \frac{\lambda_{a,i} \mB_{a,j}}{\langle B_{a},y\rangle}
	\prod_{\ell \in \A} \langle D_\ell , y \rangle^{\lambda_{\ell,i}}.
	\]
	So the matrix \(\phi = \left( \frac{\partial \phi_i}{\partial y_j} \right) \in M_{k-1,k}(\R) \) equals 
	\[
	\text{Diag}_i \left(\prod_{\ell \in \A} \langle D_\ell , y \rangle^{\lambda_{\ell,i}} \right) \cdot 
	\lambda^T \cdot \text{Diag}_a\left(1 /{\langle B_{a},y\rangle}\right) \cdot \mB.
	\]
	Let us introduce \(G \defeq 
	\left( \prod_{\ell \in \A}
	c_\ell^{1-\sum_{p=1}^{k-1}\lambda_{\ell,p}} \right) \left(\prod_{\ell \in \A} \langle B_\ell,y \rangle^{-1+\sum_{p=1}^{k-1}\lambda_{\ell,p}} \right)\).
	
	By expanding \(\det({J})\) along its last row and using Cauchy-Binet formula, we obtain
	\begin{align*}
		&  (\det{J}) / G \\
		= & \sum_{i=1}^k (-1)^{i+k} y_i \sum_{I, \lvert I\rvert = k-1} \lvert \lambda_I\rvert \cdot \lvert \mB_I^{\setminus i} \rvert \cdot \prod_{\ell \notin I} \langle B_{\ell}, y\rangle \\
		= & \sum_{i=1}^k \sum_{I, \lvert I\rvert = k-1} 
		(-1)^{i+k+\varepsilon(^cI)} y_i \gamma_{\lambda} \cdot 
		\lvert (\mAh)^{\setminus I}\rvert \cdot \lvert \mB_I^{\setminus i} \rvert \cdot \prod_{\ell \notin I} \langle B_{\ell}, y\rangle \\
		= & \sum_{i=1}^k \sum_{\sigma, \lvert \sigma\rvert = n} \sum_{\substack{u\neq v \in \A\setminus \sigma \\ \sigma u v \defeq \sigma \cup \{u\} \cup \{v\}}} 
		(-1)^{i+k+\varepsilon(\sigma u v)} (-1)^{n+\lvert \{j \in \sigma \mid j \text{ between }u\text{ and }v\} \rvert + \mathbb{1}_{u<v}} \\
		& \quad\quad\quad\quad\quad\quad
		\cdot y_i \gamma_{\lambda}  \cdot \lvert \hat{\mA}^{\sigma}\rvert \cdot h_u \cdot \lvert \mB_{\setminus \sigma u v}^{\setminus i} \rvert \cdot \prod_{\ell \in \sigma u v} \langle B_{\ell}, y\rangle.
	\end{align*}
	The last equality follows from Laplace expansion applied to the first and the last row of \(\mAh\). Let us consider \(\tilde{\mB}_{i,^c \sigma u}\) being the matrix \(\mB_{^c \sigma u}\) where the column \(i\) is replaced by the column \(\mB_{^c \sigma u} \cdot y^T\) (denoting \(\sigma\cup\{u\}\) by \(\sigma u\)). Laplace expansion along the \(i^{\text{th}}\) column also gives 
	\begin{align*}
		y_i \lvert \mB_{^c\sigma u} \rvert & = \lvert \tilde{\mB}_{i,^c\sigma u} \rvert 
		= \sum_{v \notin \sigma u} (-1)^{i+1+\varepsilon(\{v\}) +\mathbb{1}_{u<v}
		+\lvert \{j \in \sigma \mid j <v\}\rvert } \lvert \mB^{\setminus i}_{\setminus \sigma u v}\rvert \cdot \langle B_{v},y\rangle.
	\end{align*}
	Consequently, denoting by \(\mA^{u \mathbin\Vert \sigma}\) the matrix obtained by concatenating the columns of \(\mA^u\) with the columns of \(\mA^{\sigma}\)
	\begin{align*}
		& ( \det{J}) / (G \cdot \sum_{i=1}^k y_i^2) \\
		& = \sum_{\sigma, \lvert \sigma\rvert = n} \sum_{u \notin \sigma} 
		(-1)^{k+\varepsilon(\sigma u)+\lvert \{j \in \sigma \mid  j<u\} \rvert} \cdot \gamma_{\lambda}  \cdot \lvert \hat{\mA}^{\sigma}\rvert \cdot h_u \cdot \lvert \mB_{^c\sigma u} \rvert \cdot \prod_{\ell \in \sigma u} \langle B_{\ell}, y\rangle  
		\\
		& = \sum_{\sigma, \lvert \sigma\rvert = n} \sum_{u \notin \sigma} 
		(-1)^{k+\lvert \{j \in \sigma \mid  j<u\} \rvert} \cdot \gamma_{\lambda}\gamma_B  \cdot \lvert \hat{\mA}^{\sigma}\rvert \cdot h_u \cdot \lvert \mA^{\sigma u} \rvert \cdot \prod_{\ell \in \sigma u} \langle B_{\ell}, y\rangle  
		\\
		& = \sum_{\sigma, \lvert \sigma\rvert = n} \sum_{u \in \A} 
		(-1)^{k} \cdot \gamma_{\lambda}\gamma_B  \cdot \lvert \hat{\mA}^{\sigma}\rvert \cdot h_u \cdot \lvert \mA^{u \mathbin\Vert \sigma} \rvert \cdot \prod_{\ell \in \sigma u} \langle B_{\ell}, y\rangle  
		\\
		& = \sum_{t=1}^{n+1} 
		(-1)^{k+t+1}\gamma_{\lambda}\gamma_B
		\left(
		\sum_{u \in \A} 
		    \mA_{t,u} h_u \langle B_u,y \rangle
		\right)
		\left(
		\sum_{\sigma, \lvert \sigma\rvert = n} 
		  \lvert \mA^{\sigma}_{\setminus t} \rvert \cdot \lvert \hat{\mA}^{\sigma}\rvert \cdot 
		  \prod_{\ell \in \sigma} \langle B_{\ell}, y\rangle 
		\right).
	\end{align*}
	    
	The condition \(u \notin \sigma\) is removed at the third equality since \(\lvert \mA^{u \mathbin\Vert \sigma} \rvert =0\) when \(u \in \sigma\). The last equality is obtained by using Laplace expansion along the column \(u\) of the matrix \(\mA^{u \mathbin\Vert \sigma}\).
	
	By Lemma~\ref{lem_sum_sigma}, the quantity in the second parenthesis is zero if \(t \ge 2\). So
	\begin{align*}
		& ( \det{J}) / (G \cdot \sum_{i=1}^k y_i^2) 
		\\
		& =  
		(-1)^{k}\gamma_{\lambda}\gamma_B
		\left(
		\sum_{u \in \A} 
		    h_u \cdot \langle B_u,y \rangle
		\right)
		\left(
		\sum_{\sigma, \lvert \sigma\rvert = n} 
		  \lvert \hat{\mA}^{\sigma}\rvert^2 \cdot 
		  \prod_{\ell \in \sigma} \langle B_{\ell}, y\rangle 
		\right)
		\\
		& =  
		(-1)^{k}\gamma_{\lambda}\gamma_B
		\left(
		\sum_{u \in \A} 
		    h_u \cdot \langle B_u,y \rangle
		\right)
		P_{\A}(y)
		.
	\end{align*}

	That is the claimed identity.
\end{proof}

\end{document}